\documentclass[a4paper,twoside]{article}      
\usepackage{amsmath,amssymb,amsfonts,amsthm,amscd,dsfont}
\usepackage{leftidx} 
\usepackage{graphicx}                 
\usepackage{color}                    
\usepackage{mathrsfs}
\usepackage{epstopdf}
\usepackage{indentfirst}
\usepackage{enumerate}
\usepackage{url}         
\usepackage{a4wide}
\usepackage{colonequals} 
\allowdisplaybreaks

\newtheorem{theorem}{Theorem}[section]
\newtheorem{proposition}[theorem]{Proposition}

\newtheorem{lemma}[theorem]{Lemma}

\theoremstyle{definition}
\newtheorem{remark}[theorem]{Remark}

\newtheorem{assumption}{Assumption}[section]

\def\div{\mathop{\mathrm{div}}\nolimits}
\def\!{\mathop{\mathrm{!}}}

\def\R{\mathbb{ R}}

\def\E{\mathbb{ E}}

\def\U{\mathcal{U}}

\def\bRb{\mathbb{R}}
\def\bNb{\mathbb{N}}
\def\EE{\mathbb{E}}

\newtheorem{thm}{Theorem}[section]
\newtheorem{prop}[theorem]{Proposition}

\newtheorem{lem}[theorem]{Lemma}
\newtheorem{defn}[theorem]{Definition}

\def\H{\mathcal{H}}

\newlength{\boxwidth}
\title{Coupled McKean-Vlasov diffusions: wellposedness, propagation of chaos and invariant measures}

\author{Manh Hong Duong\thanks{School of Mathematics, University of Birmingham, Birmingham B15 2TT, UK (\texttt{hduong@bham.ac.uk})}
\and Julian Tugaut\thanks{Universit\'{e} Jean Monnet, Institut Camille Jordan, 23, rue du
docteur Paul Michelon, CS 82301, 42023 Saint-\'{E}tienne Cedex 2,France (\texttt{tugaut@math.cnrs.fr})}}
\begin{document}
\maketitle
\begin{abstract}
In this paper, we study a two-species model in the form of a coupled system of nonlinear stochastic differential equations (SDEs) that arises from a variety of applications such as aggregation of biological cells and pedestrian movements. The evolution of each process is influenced by four different forces, namely an external force, a self-interacting force, a cross-interacting force and a stochastic noise where the two interactions depend on the laws of the two processes. We also consider a many-particle system and a (nonlinear) partial differential equation (PDE) system that associate to the model. We prove the wellposedness of the SDEs, the propagation of chaos of the particle system, and the existence and (non)-uniqueness of invariant measures of the PDE system.  
\end{abstract}

\section{Introduction}
In this paper, we study a two-species model in the form of a coupled system of nonlinear stochastic differential equations 
\begin{subequations}
\label{eq: SDEs}
\begin{align}
&d X_t=-\nabla V_1(X_t)\,dt-a\nabla F_{11}\ast\mu_t(X_t)\,dt-(1-a)\nabla F_{12}\ast\nu_t(X_t)\,dt+\sigma\, dW_t,
\\&d Y_t=-\nabla V_2(Y_t)\,dt-a\nabla F_{21}\ast\mu_t(Y_t)\,dt-(1-a)\nabla F_{22}\ast\nu_t(Y_t)\,dt+\sigma\, d \widehat{W}_t,
\\ &\mathbb{P}(X_t\in dx)=\mu_t(x)\,dx,~~\mathbb{P}(Y_t\in dx)=\nu_t(x)\,dx.
\end{align}
\end{subequations}
Here $0\leq a\leq 1$ and $\sigma>0$ are given constants; $V_1, V_2$ are two external potentials; $F_{11}, F_{22}$ are self-interacting potentials describing the interactions among individuals of the same
species; $F_{12}, F_{21}$ are cross-interacting potentials representing the interactions between individuals belonging to different species; $\sigma$ is the diffusion intensity;  $(W_t, t\geq 0)$ and $(\widehat{W}_t, t\geq 0)$ are independent Wiener processes and finally $\ast$ denotes the standard convolution operator: for a function $G$ and a measure $\gamma$, the convolution between $G$ and $\gamma$, $G\ast\gamma$, is given by
\[
(G\ast \gamma)(x)=\int G(x-y)\gamma(y)\,dy.
\]
In \eqref{eq: SDEs} the evolution of $X_t$ and $Y_t$ depend on their own laws, $\{\mu_t,t\geq 0\}$ and $\{\nu_t, t\geq 0\}$ respectively, that are unknown. Using It\^{o} formula one can show that  $\{\mu_t,t\geq 0\}$ and $\{\nu_t, t\geq 0\}$ satisfy the following system of nonlinear nonlocal partial differential equations
\begin{subequations}
\label{eq: PDEs}
\begin{align}
&\partial_t\mu_t=\div\Big(\big(\nabla V_1+a(\nabla F_{11}\ast\mu_t)+(1-a)(\nabla F_{12}\ast \nu_t)\big)\mu_t\Big)+\frac{\sigma^2}{2}\Delta\mu_t,\\
&\partial_t\nu_t=\div\Big(\big(\nabla V_2+a(\nabla F_{21}\ast\mu_t)+(1-a)(\nabla F_{22}\ast \nu_t)\big)\nu_t\Big)+\frac{\sigma^2}{2}\Delta\nu_t,\\
&\mu_0(dx)=\mathbb{P}(X_0\in dx),~~\nu_0(dx)=\mathbb{P}(Y_0\in dx).
\end{align}
\end{subequations}
System \eqref{eq: SDEs} naturally generalizes the one-specie McKean-Vlasov dynamics
\begin{equation}
\label{eq: MKV}
d Z_t=-\nabla V(Z_t)\,dt-\nabla F\ast\zeta_t(Z_t)\,dt+\sigma\,d W_t,
\end{equation}
where $\zeta_t$ is the law of $Z_t$ that solves the following (nonlocal nonlinear) PDE
\begin{equation}
\label{eq: MKV PDE}
\partial\zeta_t=\div\Big[(\nabla V+\nabla F\ast \zeta_t)\zeta_t\Big]+\frac{\sigma^2}{2}\Delta \zeta_t.
\end{equation}
Systems of (multi-species, interacting) nonlinear stochastic differential equations and nonlocal nonlinear PDEs of the type (1)-(4) arise in a plethora of applications such as mathematical biology (bacteria chemotaxis \cite{KELLER1971,kurokiba2003,espejo2010, conca_espejo_vilches_2011,kavallaris_ricciardi_zecca_2018}, aggregation of biological cells \cite{Evers2016, Evers2017}), plasma physics and galactic dynamics~\cite{BinneyTremaine2008}, statistical mechanics and granular materials~\cite{carrillo2003,Carrillo2006}, pedestrian movements \cite{COLOMBO2012,Crippa2013}, risk management~\cite{GPY2013} and opinion formation~\cite{GPY2017}. The mathematical analysis of such systems has been getting a lot of attention over the last two decades both in the probability and in the PDE community. In particular, the McKean-Vlasov dynamics has been investigated from various aspects. Existence and uniqueness of solutions of \eqref{eq: MKV} under fairly general assumptions on the external potential $V$ and interacting potential $F$ has been proved \cite{McKean1966, Funaki1984, Sznitman1991, Meleard1996, Herrmann2008, Benachour98a,Cattiaux2008}. The propagation of chaos, which was introduced by Kac \cite{kac1956} and further developed by Sznitman \cite{Sznitman1991}, for the McKean-Vlasov was also proved~\cite{Benachour98a, Malrieu2003,Cattiaux2008}. That is, as $n$ gets large, the $n$ interacting processes
\begin{equation}
\label{eq: MKV particles}
dZ^i_t=-\nabla V(Z^i_t)\,dt-\frac{1}{n}\sum_{j=1}^n \nabla F(Z^i_t-Z^j_t)\,dt+\sigma dW^i_t,\quad i=1,\ldots, n,
\end{equation}
behave more and more like the $n$ independent processes
$$
d Z^i_t=-\nabla V(Z^i_t)\,dt-\nabla F\ast\zeta_t(Z^i_t)\,dt+\sigma\,d W^i_t, \quad i=1,\ldots, n,
$$
where $(W^i_t)_{t\geq 0}$ are independent Wiener processes and each particle’s distribution
tends to $\zeta_t(dx)=\zeta_t(x)dx$ where $\zeta_t$ solves \eqref{eq: MKV PDE}. In addition, the empirical measure $\rho^n_t:=\frac{1}{n}\sum_{j=1}^n\delta_{Z^i_t}$ converges in law, on the space $C([0,T],\R)$, to $\zeta_t(dx)$. Thus both \eqref{eq: MKV} and \eqref{eq: MKV PDE} can be numerically approximated by simulating the particle system \eqref{eq: MKV particles} for large $n$. We also refer the reader to~\cite{BolleyGuillinMalrieu2010, Duong2015NA, Jabin2016, Monmarche2017} for similar results for the Vlasov-Fokker-Planck equation, to \cite{Mischler2013,Hauray2014,Mischler2015} for analytical approach to propagation of chaos and to recent papers and surveys \cite{Jabin2017,Jabin2017TMP,DurmusEberleGuillinZimmer2018} for further discussions on this interesting topic. Another important aspect of the McKean-Vlasov dynamics, namely the existence and (non)uniqueness of invariant measures and convergence to an invariant measure, also was studied by many authors using different techniques~ \cite{Benachour98b,carrillo2003,Malrieu2003,Carrillo2006, Cattiaux2008,Bolley2013}. One interesting question that still largely remains open in this direction is to characterise the relative basins of attraction of the equilibria of the McKean-Vlasov equation when there are multiple invariant measures \cite{Dawson1983, Shiino1987, Tugaut2014}.

In contrast to the McKean-Vlasov equation, the coupled McKean-Vlasov dynamics is less understood although some initial attempts have been made. Herrmann \cite{Herrmann02} obtained results for three aforementioned issues for a special case of \eqref{eq: SDEs} where $V_1=V_2=0, F_{11}=F_{22}$ and $F_{21}=F_{12}$, see also~\cite{DuongMunteanRichardson2017} for some formal computations regarding the hydrodynamics limit for this case. Another special case, where $V_1=V_2=0$ and $\sigma=0$, has been studied by several authors: \cite{DiFrancesco2013} established a systematic existence and uniqueness theory of weak measure solutions for system \eqref{eq: PDEs} while  its equilibrium properties were investigated in \cite{Evers2017,Francesco2016}. More recently, \cite{CARLIER2016,Laborde2017} proved, using a discrete variational approximation scheme \`{a} la Jordan-Kinderlehrer-Otto, existence and uniqueness results for a class of parabolic systems with nonlinear diffusion and nonlocal interaction that includes the PDE system \eqref{eq: PDEs}. We also refer the reader to recent works \cite{LEPOUTRE2017, Chen2017,Chen2018,DIFRANCESCO2018,Carrillo2018} on similar multi-species systems where a (nonlinear) cross-diffusion is also included.

The aim of the present paper is to study the well-posedness, propagation of chaos phenomenon and the existence of (multiple) invariant measures of the coupled McKean-Vlasov system~\eqref{eq: SDEs}-\eqref{eq: PDEs}. We generalize some of the aforementioned results for special cases to the full system and obtain new results.

\textit{Well-posedness of \eqref{eq: SDEs}}. Proving the existence and uniqueness of solutions of interacting (multi-species) systems such as \eqref{eq: SDEs} is highly nontrivial because of its nonlocality and nonlinearity. When both the confinining and interaction potentials are globally Lipschitz, the well-posedness of \eqref{eq: SDEs} can be established using the by now standard techniques \cite{McKean1966, Funaki1984, Sznitman1991, Meleard1996}.  When either of the potentials is non-Lipscitz, it is a more intricate problem. The following theorem, which is our first result, generalizes similar results of~\cite{Benachour98a,Herrmann2008,PhDthesisTugaut2010} for the McKean-Vlasov equation and of \cite{Herrmann02} for the special case of \eqref{eq: SDEs} (where $V_1=V_2=0, F_{11}=F_{22}$ and $F_{21}=F_{12}$ as mentioned in a previous paragraph) to the general coupled system \eqref{eq: SDEs}.
\begin{thm}
\label{sandra}
Suppose that Assumption \ref{asp: assumption} holds and that $X_0$ and $Y_0$ are such that $\E(|X_0|^{8q^2})<\infty$ and $\E(|Y_0|^{8q^2})<\infty$ where $q>0$ is defined in (H7) of Assumption \ref{asp: assumption}. The system \eqref{eq: SDEs} admits a unique strong solution on $\bRb_+$. In other words, given a probability space with two Brownian motions, there exists a solution to the system with these Brownian motions.
\end{thm}
\textit{Propagation of chaos.}  To describe our result on propagation of chaos for the coupled McKean-Vlasov system, we take two sequences of integers, $(M_n)_{n\in\mathbb{N}}$ and $(N_n)_{n\in\mathbb{N}}$, that go to infinity as $n$ tends to infinity and consider the following system of interacting particles
\begin{subequations}
\label{eq: many SDEs}
\begin{align}
dX_t^{i}&=-\nabla V_1(X^{i}_t)\,dt-\frac{1}{N_n+M_n}\sum_{j=1}^{N_n}\nabla F_{11}(X^{i}_t-X^{j}_t)\,dt\notag
\\&\qquad-\frac{1}{N_n+M_n}\sum_{k=1}^{M_n}\nabla F_{12}(X^{i}_t-Y^{k}_t)\,dt+\sigma dW^i_t;~~i=1,\ldots, N_n;
\\dY^{i}_t&=-\nabla V_2(Y^{i}_t)\,dt-\frac{1}{N_n+M_n}\sum_{j=1}^{N_n}\nabla F_{21}(Y^{i}_t-X^{j}_t)\,dt\notag
\\&\qquad-\frac{1}{N_n+M_n}\sum_{k=1}^{M_n}\nabla F_{22}(Y^{i}_t-Y^{k}_t)\,dt+\sigma d\widetilde W^i_t,~~i=1,\ldots, M_n.
\end{align}
\end{subequations}
Note that the convolution operators $\nabla F_{ij}\ast \gamma$ ($i,j\in\{1,2\}, \gamma\in\{\mu,\nu\})$ in \eqref{eq: SDEs} are replaced by the average sums in \eqref{eq: many SDEs}. These sums can also be viewed as convolutions between $\nabla F_{ij}$ with the empirical measures, $\mu^n_t$ and $\nu^n_t$, instead of the laws $\mu_t$ and $\nu_t$ where
\[
\mu^{n}_t:=\frac{1}{M_n+N_n}\sum_{j=1}^{N_n}\delta_{X^j_t}\quad \text{and}\quad\nu^{n}_t:=\frac{1}{M_n+N_n}\sum_{k=1}^{M_n}\delta_{Y^k_t}.
\]
We will show that the propagation of chaos phenomenon holds for the system \eqref{eq: many SDEs}, that is, for all $(p,q)\in \mathbb{N}^2$, $(X^1_t,\ldots, X^p_t,Y^1_t,\ldots, Y^q_t)$ converges as $n$ tends to infinity to $\otimes_{i=1}^p\mu_t\otimes_{j=1}^q \nu_t$, where $\mu_t, \nu_t$ are respectively the laws of $X_t$ and $Y_t$ that are solutions of \eqref{eq: SDEs}. This result is the consequence of the following theorem, that is our second result and extends \cite{Herrmann02} to the general case,
\begin{thm}
\label{theo: PoC}
Under the same assumption as in Theorem \ref{sandra}, for $T<\infty$, we have
\begin{equation}
\label{eq: E-sup}
\lim\limits_{n\to\infty}\E\Big[\sup\limits_{t\in[0,T]}\big(X_t^{i}-\widehat{X_t^{i}}\big)^2\Big]=0\quad\text{and}\quad \lim\limits_{n\to\infty}\E\Big[\sup\limits_{t\in[0,T]}\big(Y_t^{i}-\widehat{Y_t^{i}}\big)^2\Big]=0,
\end{equation}
where $(\widehat{X_t^{i}}, \widehat{Y_t^{i}})$ is a solution to the following system
\begin{subequations}
\label{eq: MV-system}
\begin{align}
d\widehat{X_t^{i}}&=-\nabla V_1(\widehat{X^{i}_t})\,dt-a (\nabla F_{11}\ast\mu_t)(\widehat{X_t^{i}})\,dt\notag
\\&\qquad\qquad-(1-a)(\nabla F_{12}\ast\nu_t)(\widehat{X_t^{i}})\,dt+\sigma dW^i_t, \quad i=1,\ldots, N_n;
\\ d\widehat{Y^i_t}&=-\nabla V_2(\widehat{Y_t^{i}})\,dt-a(\nabla F_{21}\ast\mu_t)(\widehat{Y_t^{i}})\,dt\notag
\\&\quad\qquad-(1-a)(\nabla F_{22}\ast\nu_t)(\widehat{Y_t^{i}})\,dt+\sigma d\widetilde W^i_t,
\quad i=1,\ldots, M_n,
\end{align}
with $(W^i_t,\widetilde W^i_t)$ being independent Wiener processes. Note that $\{\widehat{X_t}^i\}_{i=1}^{N_n}$ ($\{Y_t^k\}_{k=1}^{M_n}$ resp.) are identically independent copies of $X_t$ ($Y_k$ resp.).
\end{subequations}
\end{thm}

\textit{Existence and non-uniqueness of invariant measures in non-convex landscapes.} It is by now well-known that when the confining potential $V$ is not convex the McKean-Vlasov equation exhibits a phase transition phenomenon, that is it may have a unique stationary solution or several ones when the diffusion coefficient (i.e., the temperature) is above or below a critical value \cite{Dawson1983,Tamura1984, Shiino1987, Tugaut14, BCCD16}.  Similar results of nonuniqueness of the stationary state at low temperatures have been also obtained for McKean-Vlasov equations modeling opinion formation~\cite{Chazelle_al2017, ChayesPanferov2010}, for the Desai-Zwanzig model in a two-scale potential~\cite{GomesPavliotis2018} as well as for the McKean-Vlasov equations on the torus \cite{ChayesPanferov2010,CGPS2018tmp}. Our third result is the following existence and non-uniqueness of invariant measures. This is significantly different from \cite{Herrmann02} where there is a unique invariant measure.
\begin{thm}
\label{theo: invariant}
Suppose that $F_{ij}(x)=\frac{\alpha_{ij} x^2}{2}$ for $i,j\in\{1,2\}$ and that $V_1$ and $V_2$ have a common unique minimizer $m^\ast$. Then for any $\rho$ such that 
\begin{equation*}
\rho\geq \max\Big\{\frac{|V_1^{(3)}(m^\ast)|}{4V_1''(m^\ast)(V_1''(m^\ast)+a\alpha_{11}+(1-a)\alpha_{12})},\frac{|V_2^{(3)}(m^\ast)|}{4V_2''(m^\ast)(V_2''(m^\ast)+a\alpha_{21}+(1-a)\alpha_{22})} \Big\}.
\end{equation*} 
the system \eqref{eq: PDEs} have an invariant measure $(\mu,\nu)$ whose mean values belong to $[m^\ast-\rho\sigma^2,m^\ast+\rho\sigma^2]\times[m^\ast-\rho\sigma^2,m^\ast+\rho\sigma^2]$. In addition, if $V_1$ and $V_2$ are symmetrical, then there is a unique symmetrical invariant measure $(\mu^0,\nu^0)$ whose mean values are zeros.
\end{thm}
This implies that if $V_1=V_2=V$ where $V$ is a double-wells landscape, then there are at least three invariant probabilities.\\[4pt]
\textit{Organisation of the paper.} The rest of the paper is organised as follows. In Section \ref{sec: wellpose}, we prove Theorem \ref{sandra} on the wellposedness of \eqref{eq: SDEs}. In Section \ref{sec: PoC} we study the propagation of chaos phenomenon and establish Theorem \ref{theo: PoC}. Finally, in Section \ref{sec: invariant} we prove Theorem \ref{theo: invariant} on the existence and nonuniqueness of invariant measures.
\section{Existence and uniqueness of strong solutions}
\label{sec: wellpose}

In this section, we prove Theorem \ref{sandra} establishing the existence and unique of strong solutions of the system \eqref{eq: SDEs}. We adapt the proof of  \cite{Benachour98a} for the existence and uniqueness of strong solutions of the McKean-Vlasov dynamics \eqref{eq: MKV}, see also \cite{Herrmann02,Herrmann2008,PhDthesisTugaut2010}. To this end, we transform \eqref{eq: SDEs} into a fixed point problem of a map $\Gamma$ on a functional space $\Lambda$, we then show that $\Gamma$ is a contraction map on a subspace $\Lambda_T\subset\Lambda$ 
proving the existence and uniqueness of strong solutions over a finite time interval $[0,T]$. The local solution is then extended to become a global one by controlling its moments.
\begin{assumption}We make the following assumptions.
\label{asp: assumption}
\begin{enumerate}[(H1)]
\item The coefficients $\nabla V_1$, $\nabla V_2$, $\nabla F_{ij}$ are locally Lipschitz for any $i,j\in\{1;2\}$.
\item The functions $V_1$, $V_2$ and $F_{ij}$ are continuously differentiable for any $i,j\in\{1;2\}$.
\item There exist $\theta_1>0$ and $\theta_2>0$ such that
\begin{equation}
\label{eq: V1}
(\nabla V_1(x)-\nabla V_1(y))(x-y)\geq -\theta_1|x-y|^2\quad\text{and}\quad (\nabla V_2(x)-\nabla V_2(y))(x-y)\geq -\theta_2|x-y|^2\quad\forall x,y.
\end{equation}
\item $xV_1'(x)\geq C_4x^4-C_2x^2$ with $C_2,C_4>0$. The same holds with $V_2$.
\item The potentials $V_1$ is convex at infinity: $\displaystyle\lim_{|x|\to+\infty}\nabla^2V_1(x)=+\infty$. The same holds with $V_2$.
\item There exist $m\in\bNb$ and $C>0$ such that $|\nabla V_1(x)|+|\nabla V_2(x)|\leq C|x|^{2m-1}$ and $m\geq2$.
\item $\nabla F_{11}$ and $\nabla F_{22}$ are odd and increasing with polynomial growth functions, the degree being $2q-1$.
\item $\nabla F_{12}$ and $\nabla F_{21}$ are Lipschitz.
\end{enumerate}
\end{assumption}

 We now need to introduce some functional spaces.

\begin{defn}
\label{sandra2}
On the space of functions from $\bRb_+\times\bRb$ to $\bRb$, we introduce the norm
\begin{equation*}
\left|\left|b\right|\right|_T:=\sup_{0\leq s\leq T}\sup_{x\in\bRb}\left(\frac{\left|b(s,x)\right|}{1+|x|^{2q}}\right)\,.
\end{equation*}
\end{defn}

We now introduce the functional space that will be used in the following.

\begin{defn}
\label{sandra3}
We consider the space
\begin{equation*}
\Lambda_T:=\Lambda_T^1\bigcap\Lambda_T^2\bigcap\Lambda_T^3\,,
\end{equation*}
where the three spaces of functions $\Lambda_T^1$, $\Lambda_T^2$ and $\Lambda_T^3$ are defined by
\begin{equation*}
\Lambda_T^1:=\left\{b:[0;T]\times\bRb\longrightarrow\bRb\,\,\left|\right.\,\,x\mapsto b(s,x)\mbox{ is locally Lipschitz uniformly in }s\right\}\,,
\end{equation*}
where the parameter of Lipschitz may depend on $b$;
\begin{equation*}
\Lambda_T^2:=\left\{b:[0;T]\times\bRb\longrightarrow\bRb\,\,\left|\right.\,\,x\mapsto b(s,x)\mbox{ is increasing and }b(s,x)-b(s,y)\geq\xi_1(x-y)+\xi_0\right\}\,,
\end{equation*}
where $\xi_1>0$, $\xi_0\in\bRb$ and $x\geq y$; and
\begin{equation*}
\Lambda_T^3:=\left\{b:[0;T]\times\bRb\longrightarrow\bRb\,\,\left|\right.\,\,\left|\left|b\right|\right|_T<\infty\right\}\,.
\end{equation*}
The space $\Lambda_T$  is equipped with the norm $||.||_T$.
\end{defn}

\begin{defn}
\label{topchef}
We finally put $F_T:=\Lambda_T\times\Lambda_T\times\Lambda_T\times\Lambda_T$ equipped with the norm 
\begin{equation*}
\left|\left|b\right|\right|_T^F:=\sum_{i=1}^4\left|\left|b_i\right|\right|_T\,,
\end{equation*}
where $b:=(b_1,b_2,b_3,b_4)$.
\end{defn}

We will also use a transformation in order to apply a fixed point theorem.

\begin{defn}
\label{huile}
We consider $\Gamma$ from $F_T$ to $F_T$ defined by its coordinates:
\begin{align*}
&p_1\circ\Gamma(b)(x):=a\EE\left[\nabla F_{11}\left(x-X_t^{b}\right)\right]\,,\quad p_2\circ\Gamma(b)(x):=(1-a)\EE\left[\nabla F_{12}\left(x-Y_t^{b}\right)\right]\,,\\
&p_3\circ\Gamma(b)(x):=a\EE\left[\nabla F_{21}\left(x-X_t^{b}\right)\right]\quad\mbox{and}\quad p_4\circ\Gamma(b)(x):=(1-a)\EE\left[\nabla F_{22}\left(x-Y_t^{b}\right)\right]\,,
\end{align*}
where $p_i$ is the $i$th projection on the space $F_T$ and $X_t^{b}$ (resp. $Y_t^{b}$) is solution of the SDE
\begin{equation}
\label{hanae}
dX_t^{b}=\sigma dB_t-\nabla V_1\left(X_t^{b}\right)dt-b_1\left(t,X_t^{b}\right)dt-b_2\left(t,X_t^{b}\right)dt\,,
\end{equation}
respectively
\begin{equation}
\label{hanae2}
dY_t^{b}=\sigma d\widetilde{B_t}-\nabla V_2\left(Y_t^{b}\right)dt-b_3\left(t,Y_t^{b}\right)dt-b_4\left(t,Y_t^{b}\right)dt\,.
\end{equation}
\end{defn}

To show that there exist solutions to the equations on $X^{b}$ and on $Y^{b}$, we use the following result (see \cite[Theorem 10.2.2]{Stroock} at page 255):

\begin{prop}
\label{chap2:prop:existence:nonlin}
Let $b$ : $\bRb_+\times\mathbb{R}\longrightarrow\mathbb{R}$ be a function satisfying the three following properties:
\begin{enumerate}
\item  $\max_{s\geq 0}|b(s,0)|<\infty$.
\item  For any $n\in\mathbb{N}$, there exists a constant $c_n>0$ such that $|b(s,x)-b(s,y)|\leq c_n|x-y|$ for any reals $x$ and $y$ satisfying $|x|<n$ and $|y|<n$.
\item There exists a constant $r>0$ such that for any $|x|>r$, ${\rm sgn}(x)b(s,x)\geq 0$.
\end{enumerate}
Then, for any random variable $X_0$, the equation $E^{(b,X_0)}$ admits a unique strong solution where $E^{(b,X_0)}$ is defined by
\begin{equation*}
X_t=X_0-\int_0^tb(s,X_s)ds+\sigma B_t\,.
\end{equation*}
\end{prop}

To show that there is a unique strong solution to the initial system, we search a fixed point to the transformation $\Gamma$. To do so, it is easy to check that for any $b\in F_T$, the equations \eqref{hanae} and \eqref{hanae2} admit a unique strong solution. Indeed, the convexity at infinity of the potentials $V_1$ and $V_2$ guarantees that the third point of Proposition \ref{chap2:prop:existence:nonlin} is satisfied.

The following definition of moments will play a crucial role in the analysis of this paper.

\begin{defn}
\label{moments}
For any $b\in F_T$ and $p>0$, we define
\begin{align*}
&\eta_p^{b}(t):=\EE\left[\left|X_t^{b}\right|^p\right]\,\,,\quad\widehat{\eta_p^{b}}(t):=\sup_{0\leq s\leq t}\eta_p^{b}(s)\,,\\
&\xi_p^{b}(t):=\EE\left[\left|Y_t^{b}\right|^p\right]\quad\mbox{and}\quad\widehat{\xi_p^{b}}(t):=\sup_{0\leq s\leq t}\xi_p^{b}(s)\,.
\end{align*}
\end{defn}

To prove Theorem \ref{sandra}, we need several lemmas.

\begin{lem}
\label{oumaima}
Set $b\in F_T$, $n\geq1$, $\rho:=(\rho_0,\rho_0,\rho_0,\rho_0)$ with $\rho_0(x):=\beta_0 x$, then $\widehat{\eta_{2n}^{\rho}}(T)+\widehat{\xi_{2n}^{\rho}}<\infty$, for $n\geq0$ such that $\EE\left[\left|X_0^{2n}\right|\right]<\infty$ and $\EE\left[\left|Y_0^{2n}\right|\right]<\infty$. Moreover:
\begin{equation*}
\widehat{\eta_{2n}^{b}}(T)\leq k_1(n)\left[T^{2n}+\left(||b_1-\rho_0||_T^{2n}+||b_2-\rho_0||_T^{2n}\right)\left(T^{2n}+\widehat{\eta_{4qn}^{\rho}}(T)\right)\right]\,,
\end{equation*}
and
\begin{equation*}
\widehat{\xi_{2n}^{b}}(T)\leq k_2(n)\left[T^{2n}+\left(||b_3-\rho_0||_T^{2n}+||b_4-\rho_0||_T^{2n}\right)\left(T^{2n}+\widehat{\xi_{4qn}^{\rho}}(T)\right)\right]\,,
\end{equation*}
where $k_1(n)$ and $k_2(n)$ are constants which do not depend on $b$, $\rho$ or $T$.

\end{lem}

\begin{proof}
{\bf Step 1.} By considering $\rho:=(\rho_0,\rho_0,\rho_0,\rho_0)$, we thus have the following equation:
\begin{equation*}
X_t^\rho=X_0+\sigma B_t-\int_0^t\nabla V_1\left(X_s^\rho\right)ds-2\beta_0\int_0^tX_s^\rho ds\,.
\end{equation*}
For any $n\geq1$, It\^o formula yields
\begin{align*}
\frac{d}{dt}\EE\left[\left(X_t^\rho\right)^{2n}\right]=&n(2n-1)\sigma^2\EE\left[\left(X_t^\rho\right)^{2n-2}\right]-4n\beta_0\EE\left[\left(X_t^\rho\right)^{2n}\right]-2n\EE\left[\left(X_t^\rho\right)^{2n-1}V'\left(X_t^\rho\right)\right]\\
\leq&n(2n-1)\sigma^2\left(\EE\left[\left(X_t^\rho\right)^{2n}\right]\right)^{1-\frac{1}{n}}-4n\beta_0\EE\left[\left(X_t^\rho\right)^{2n}\right]\\
&-2nC_4\EE\left[\left(X_t^\rho\right)^{2n+2}\right]+2nC_2\EE\left[\left(X_t^\rho\right)^{2n}\right]\\
\leq&n(2n-1)\sigma^2\left(\EE\left[\left(X_t^\rho\right)^{2n}\right]\right)^{1-\frac{1}{n}}+2n(C_2-2\beta_0)\EE\left[\left(X_t^\rho\right)^{2n}\right]\\
&-2nC_4\left(\EE\left[\left(X_t^\rho\right)^{2n}\right]\right)^{1+\frac{1}{n}}\,.
\end{align*}
We immediately deduce that
\begin{equation*}
\EE\left[\left(X_t^\rho\right)^{2n}\right]\leq\max\left\{\EE\left[\left(X_0\right)^{2n}\right]\,;\,\left(\frac{C_2-2\beta_0+\sqrt{(C_2-2\beta_0)^2+2C_4(2n-1)\sigma^2}}{2C_4}\right)^{n}\right\}<\infty\,.
\end{equation*}
We deduce that
\begin{equation*}
\sup_{t\geq0}\EE\left[\left(X_t^\rho\right)^{2n}\right]\leq c_n\left(1+\EE\left[\left(X_0\right)^{2n}\right]\right)\,,
\end{equation*}
where $c_n$ is constant. Similarly we obtain 
\begin{equation*}
\sup_{t\geq0}\EE\left[\left(Y_t^\rho\right)^{2n}\right]\leq c_n\left(1+\EE\left[\left(Y_0\right)^{2n}\right]\right)\,.
\end{equation*}
As a consequence, if $\EE\left[\left(X_0\right)^{2n}\right]$ and $\EE\left[\left(Y_0\right)^{2n}\right]$ are finite, we have that $\widehat{\eta_{2n}^{\rho}}(T)<\infty$ and $\widehat{\xi_{2n}^{\rho}}(T)<\infty$.\\[4pt]
{\bf Step 2.} We have
\begin{align*}
X_t^{b}-X_t^\rho=&-\int_0^t\left[V_1'\left(X_s^{b}\right)-V_1'\left(X_s^\rho\right)\right]ds\\
&-\int_0^t\left[b_1\left(s,X_s^{b}\right)-p_1\circ\rho\left(s,X_s^\rho\right)\right]ds-\int_0^t\left[b_2\left(s,X_s^{b}\right)-p_2\circ\rho\left(s,X_s^\rho\right)\right]ds\,.
\end{align*}
Consequently, for any $\alpha>1$, we obtain that $\left|X_t^{b}-X_t^\rho\right|^\alpha$ is equal to
\begin{align*}
&-\alpha\int_0^t{\rm sign}\left(X_s^{b}-X_s^\rho\right)\left|X_s^{b}-X_s^\rho\right|^{\alpha-1}\mathds{1}_{X_s^{b}\neq X_s^\rho}\left[V_1'\left(X_s^{b}\right)-V_1'\left(X_s^{\rho}\right)\right]ds\\
&-\alpha\int_0^t{\rm sign}\left(X_s^{b}-X_s^\rho\right)\left|X_s^{b}-X_s^\rho\right|^{\alpha-1}\mathds{1}_{X_s^{b}\neq X_s^\rho}\left[b_1\left(s,X_s^{b}\right)-p_1\circ\rho\left(s,X_s^{\rho}\right)\right]ds\\
&-\alpha\int_0^t{\rm sign}\left(X_s^{b}-X_s^\rho\right)\left|X_s^{b}-X_s^\rho\right|^{\alpha-1}\mathds{1}_{X_s^{b}\neq X_s^\rho}\left[b_2\left(s,X_s^{b}\right)-p_2\circ\rho\left(s,X_s^{\rho}\right)\right]ds\,.
\end{align*}
Taking the limit as $\alpha$ goes to $1^+$ we get
\begin{align}
\label{eq: difference}
\left|X_t^{b}-X_t^\rho\right|=&-\int_0^t{\rm sign}\left(X_s^{b}-X_s^\rho\right)\left[V_1'\left(X_s^{b}\right)-V_1'\left(X_s^{\rho}\right)\right]ds\notag\\
&-\int_0^t{\rm sign}\left(X_s^{b}-X_s^\rho\right)\left[b_1\left(s,X_s^{b}\right)-p_1\circ\rho\left(s,X_s^{\rho}\right)\right]ds\notag\\
&-\int_0^t{\rm sign}\left(X_s^{b}-X_s^\rho\right)\left[b_2\left(s,X_s^{b}\right)-p_2\circ\rho\left(s,X_s^{\rho}\right)\right]ds\,.
\end{align}
We will control each term on the right-hand side of \eqref{eq: difference}. The first one can be controlled by
\begin{equation*}
-\int_0^t{\rm sign}\left(X_s^{b}-X_s^\rho\right)\left[V_1'\left(X_s^{b}\right)-V_1'\left(X_s^{\rho}\right)\right]ds\leq-\gamma\int_0^t\left|X_s^{b}-X_s^\rho\right|ds+T\widetilde{\gamma}\,,
\end{equation*}
for any $t\leq T$. In the last formula, $\gamma$ and $\widetilde{\gamma}$ are constants which depend on $V_1$. For the second term: since $b_1$ is increasing
\begin{equation*}
\rm sign \left(X_s^{b}-X_s^\rho\right)\left[b_1\left(s,X_s^{b}\right)-b_1\left(s,X_s^{\rho}\right)\right]\geq 0,
\end{equation*}
which implies that
\begin{align*}
&-\rm sign\left(X_s^{b}-X_s^\rho\right)\left[b_1\left(s,X_s^{b}\right)-p_1\circ\rho\left(s,X_s^{\rho}\right)\right]
\\&\leq -\rm sign\left(X_s^{b}-X_s^\rho\right)\left[b_1\left(s,X_s^{b}\right)-p_1\circ\rho\left(s,X_s^{\rho}\right)\right]+\rm sign \left(X_s^{b}-X_s^\rho\right)\left[b_1\left(s,X_s^{b}\right)-b_1\left(s,X_s^{\rho}\right)\right]
\\&=-\rm sign\left(X_s^{b}-X_s^\rho\right)\left[b_1\left(s,X_s^{\rho}\right)-p_1\circ\rho\left(s,X_s^{\rho}\right)\right]
\\&\leq |b_1\left(s,X_s^{\rho}\right)-p_1\circ\rho\left(s,X_s^{\rho}\right)|.
\end{align*}
As a consequence, the second term is bounded above by
$$
\int_0^t  |b_1\left(s,X_s^{\rho}\right)-p_1\circ\rho\left(s,X_s^{\rho}\right)|\,ds.
$$
Similarly the third term is bounded above by
$$
\int_0^t  |b_2\left(s,X_s^{\rho}\right)-p_2\circ\rho\left(s,X_s^{\rho}\right)|\,ds.
$$
Substituting these estimates back into \eqref{eq: difference} we get
\begin{align*}
\left|X_t^{b}-X_t^\rho\right|\leq&T\widetilde{\gamma}+\int_0^t\left|b_1\left(s,X_s^{\rho}\right)-p_1\circ\rho\left(s,X_s^{\rho}\right)\right|ds\\
&+\int_0^t\left|b_2\left(s,X_s^{\rho}\right)-p_2\circ\rho\left(s,X_s^{\rho}\right)\right|ds\,.
\end{align*}
As $b_1$, $b_2$ and $\rho$ are in the space $\Lambda_T$, we know that $||b_1||_T+||b_2||_T+||\rho||_T<\infty$ so that $||b_1-\rho||_T<\infty$ and $||b_2-\rho||_T<\infty$. We directly deduce:
\begin{equation*}
\left|X_t^{b}-X_t^\rho\right|\leq T\widetilde{\gamma}+\left(||b_1-\rho||_T+||b_2-\rho||_T\right)\int_0^t\left(1+\left(X_s^{\rho}\right)^{2q}\right)ds\,.
\end{equation*}
By using triangular inequality, we obtain:
\begin{equation*}
\left|X_t^{b}\right|^{2n}\leq\left(\left|X_t^{\rho}\right|+\left|X_t^{b}-X_t^\rho\right|\right)^{2n}\leq2^{2n}\left\{\left|X_t^{\rho}\right|^{2n}+\left|X_t^{b}-X_t^\rho\right|^{2n}\right\}\,.
\end{equation*}
Consequently, we have:
\begin{equation*}
\widehat{\eta_{2n}^{b}}(T)=\sup_{0\leq t\leq T}\EE\left[\left|X_t^{b}\right|^{2n}\right]\leq2^{2n}\left(\widehat{\eta_{2n}^{\rho}}(T)+\sup_{0\leq t\leq T}\EE\left[\left|X_t^{b}-X_t^\rho\right|^{2n}\right]\right)\,.
\end{equation*}
But, we can write
\begin{align*}
\left|X_t^{b}-X_t^\rho\right|^{2n}\leq&2^{2n}\left\{T^{2n}\widetilde{\gamma}^{2n}+2^{2n}\left(||b_1-\rho||_T^{2n}+||b_2-\rho||_T^{2n}\right)\left[\int_0^T\left(1+\left|X_t^{\rho}\right|^{2q}\right)dt\right]^{2n}\right\}\\
\leq&2^{2n}\left\{T^{2n}\widetilde{\gamma}^{2n}+2^{4n}\left(||b_1-\rho||_T^{2n}+||b_2-\rho||_T^{2n}\right)\left[T^{2n}+\left(\int_0^T\left|X_t^{\rho}\right|^{2q}dt\right)^{2n}\right]\right\}\\
\leq&2^{2n}\left\{T^{2n}\widetilde{\gamma}^{2n}+2^{4n}\left(||b_1-\rho||_T^{2n}+||b_2-\rho||_T^{2n}\right)\left[T^{2n}+\int_0^T\left|X_t^{\rho}\right|^{4qn}dt\right]\right\}\,.
\end{align*}
By taking the expectation then the supremum over $[0;T]$, we find the formula for $\widehat{\eta_{2n}^{b}}(T)$. The same computations hold for the second diffusion.

\end{proof}

\begin{lem}
\label{vin}
$\Gamma$ is an application from $F_T$ to $F_T$ and
\begin{equation}
\label{eq:vin}
\left|\left|\Gamma b\right|\right|_T^F\leq C_0\left(1+\widehat{\eta_{2q}^{b}}(T)+\widehat{\xi_{2q}^{b}}(T)\right)\,,
\end{equation}
where $C_0$ is a positive constant.
\end{lem}
\begin{proof}
{\bf Step 1.} We first need to prove that $p_i\circ\Gamma(b)$ lives in $\Lambda_T^1\bigcap\Lambda_T^2$ for any $1\leq i\leq 4$. We will do so only for $i=1$. As $\nabla F_{11}$ is increasing and continuous, we deduce that $p_1\circ\Gamma(b)$ is continuous and increasing in $x$. It is also locally Lipschitz (uniformly in the time variable). Due to the assumptions on the potential $F_{11}$, we have for any $x\geq y$
\begin{align*}
p_1\circ\Gamma(b)(t,x)-p_1\circ\Gamma(b)(t,y)=&a\EE\left[\nabla F_{11}\left(x-X_t^{b}\right)-\nabla F_{11}\left(y-X_t^{b}\right)\right]\\
\geq&a\beta_{11}^1(x-y)+a\beta_{11}^0\,.
\end{align*}
By taking $\xi_1:=\inf\left\{a\beta_{11}^1;a\beta_{21}^1;(1-a)\beta_{12}^1;(1-a)\beta_{22}^1\right\}$ and $\xi_0:=\inf\left\{a\beta_{11}^0;a\beta_{21}^0;(1-a)\beta_{12}^0;(1-a)\beta_{22}^0\right\}$ we obtain that $p_i\circ\Gamma(b)$ is in $\Lambda_T^1\bigcap\Lambda_T^2$.\\[4pt]
{\bf Step 2.} We will now prove Inequality \eqref{eq:vin} (which, by the way, proves that $p_i\circ\Gamma(b)$ lives in $\Lambda_T^3$). By definition, we have:
\begin{align}
\nonumber
\left|\left|p_1\circ\Gamma(b)\right|\right|_T:=&a\sup_{x\in\bRb}\frac{\left|\EE\left[\nabla F_{11}\left(x-X_t^{b}\right)\right]\right|}{1+x^{2q}}\\
\nonumber
\leq&a\sup_{x\in\bRb}\frac{\EE\left[\left|\nabla F_{11}\left(x-X_t^{b}\right)\right|\right]}{1+x^{2q}}\\
\nonumber
\leq&a\sup_{x\in\bRb}\frac{C\left(1+|x|^{2q}+\EE\left[\left|X_t^{b}\right|^{2q}\right]\right)}{1+x^{2q}}\\
\label{chaudron}
\leq&aC\left(1+\widehat{\eta_{2q}^{b}}(T)\right)\,.
\end{align}
By proceeding similarly, we obtain
\begin{align}
\label{chaudron2}
&\left|\left|p_2\circ\Gamma(b)\right|\right|_T\leq (1-a)C\left(1+\widehat{\xi_{2q}^{b}}(T)\right)\,,\\
\label{chaudron3}
&\left|\left|p_3\circ\Gamma(b)\right|\right|_T\leq aC\left(1+\widehat{\eta_{2q}^{b}}(T)\right)\,,\\
\label{chaudron4}
\mbox{and}\quad&\left|\left|p_4\circ\Gamma(b)\right|\right|_T\leq (1-a)C\left(1+\widehat{\xi_{2q}^{b}}(T)\right)\,.
\end{align}

As a consequence, we have
\begin{equation*}
\left|\left|\Gamma b\right|\right|_T^F\leq C_0\left(1+\widehat{\eta_{2q}^{b}}(T)+\widehat{\xi_{2q}^{b}}(T)\right)\,.
\end{equation*}

\end{proof}

\begin{lem}
\label{vin2}
$\Gamma$ is continuous and satisfies
\begin{align}
\label{lala1}
&\left|\left|p_1\circ\Gamma\left(b\right)-p_1\circ\Gamma\left(c\right)\right|\right|_T\leq\left(||b_1-c_1||_T+||b_2-c_2||_T\right)\sqrt{T}C_0'\left(\widehat{\eta_{4q}^{b}}(T),\widehat{\eta_{4q}^{c}}(T)\right)\,,\\
\label{lala2}
&\left|\left|p_2\circ\Gamma\left(b\right)-p_2\circ\Gamma\left(c\right)\right|\right|_T\leq\left(||b_1-c_1||_T+||b_2-c_2||_T\right)\sqrt{T}C_0'\left(\widehat{\eta_{4q}^{b}}(T),\widehat{\eta_{4q}^{c}}(T)\right)\,,\\
\label{lala3}
&\left|\left|p_3\circ\Gamma\left(b\right)-p_3\circ\Gamma\left(c\right)\right|\right|_T\leq\left(||b_3-c_3||_T+||b_4-c_4||_T\right)\sqrt{T}C_0'\left(\widehat{\xi_{4q}^{b}}(T),\widehat{\xi_{4q}^{c}}(T)\right)\,,\\
\label{lala4}
\mbox{and}\quad&\left|\left|p_4\circ\Gamma\left(b\right)-p_4\circ\Gamma\left(c\right)\right|\right|_T\leq\left(||b_3-c_3||_T+||b_4-c_4||_T\right)\sqrt{T}C_0'\left(\widehat{\xi_{4q}^{b}}(T),\widehat{\xi_{4q}^{c}}(T)\right)\,,
\end{align}
where $C_0'$ is an increasing function for both variables.

\end{lem}
\begin{proof}
Set $s\in[0;T]$ and $x\in\bRb$. By triangular inequality, we have
\begin{equation*}
\left|p_1\circ\Gamma\left(b\right)(s,x)-p_1\circ\Gamma\left(c\right)(s,x)\right|\leq\EE\left[\left|F_{11}'\left(x-X_s^{b}\right)-F_{11}'\left(x-X_s^{c}\right)\right|\right]\,.
\end{equation*}
By the assumptions on $F_{11}$, we get:
\begin{align*}
&\left|p_1\circ\Gamma\left(b\right)(s,x)-p_1\circ\Gamma\left(c\right)(s,x)\right|\leq C_q\EE\left[\Delta_s(b,c)\left(1+\left(\Delta_s^{b}(x)\right)^{2q-2}+\left(\Delta_s^{c}(x)\right)^{2q-2}\right)\right],\\
\mbox{where}\quad&\Delta_s(b,c):=\left|X_s^{b}-X_s^{c}\right|,\\
\mbox{and}\quad&\Delta_s^{b}(x):=\left|x-X_s^{b}\right|\quad\mbox{for any }b\in F_T\,.
\end{align*}
As $(a+b)^{2q-2}\leq2^{2q-2}\left(a^{2q-2}+b^{2q-2}\right)$, we deduce:
\begin{align*}
&\left|p_1\circ\Gamma\left(b\right)(s,x)-p_1\circ\Gamma\left(c\right)(s,x)\right|\leq 2^{2q}C_q\left(1+x^{2q-2}\right)\EE\left[\Delta_s(b,c)\left(1+\left(X_s^{b}\right)^{2q-2}+\left(X_s^{c}\right)^{2q-2}\right)\right]
\end{align*}
We remind that $(a+b+c)^2\leq3(a^2+b^2+c^2)$ and $1+x^{2q-2}\leq2(1+x^{2q})$. Then, Cauchy-Schwarz inequality yields
\begin{align}
\label{nadege}
\left|p_1\circ\Gamma\left(b\right)(s,x)-p_1\circ\Gamma\left(c\right)(s,x)\right|\leq&3\times2^{2q+1}C_q\left(1+|x|^{2q}\right)\sqrt{\EE\left[\left(\Delta_s(b,c)\right)^2\right]}\\
&\times\sqrt{1+\widehat{\eta_{4q-4}^{b}}(T)+\widehat{\eta_{4q-4}^{c}}(T)}\,.\nonumber
\end{align}

By using It\^o formula with the function $x\mapsto|x|^2$, we can write
\begin{align*}
\Delta_t(b,c)^2=&-2\int_0^t\left(X_s^{b}-X_s^{c}\right)\left(V_1'\left(X_s^{b}\right)-V_1'\left(X_s^{c}\right)\right)ds\\
&-2\int_0^t\left(X_s^{b}-X_s^{c}\right)\left(b_1\left(s,X_s^{b}\right)-c_1\left(s,X_s^{c}\right)\right)ds\\
&-2\int_0^t\left(X_s^{b}-X_s^{c}\right)\left(b_2\left(s,X_s^{b}\right)-c_2\left(s,X_s^{c}\right)\right)ds\,.
\end{align*}
The first term is less than $2\theta\int_0^t\Delta_s(b,c)^2ds$. Since the functions $b_1$ and $b_2$ are increasing, we deduce that the quantities $\left(X_s^{b}-X_s^{c}\right)\left(b_1\left(s,X_s^{b}\right)-b_1\left(s,X_s^{c}\right)\right)$ and $\left(X_s^{b}-X_s^{c}\right)\left(b_2\left(s,X_s^{b}\right)-b_2\left(s,X_s^{c}\right)\right)$ are nonnegative. This implies
\begin{align*}
&-\int_0^t\left(X_s^{b}-X_s^{c}\right)\left(b_1\left(s,X_s^{b}\right)-c_1\left(s,X_s^{c}\right)\right)ds\\
\leq&\int_0^t\left|X_s^{b}-X_s^{c}\right|\left|b_1\left(s,X_s^{c}\right)-c_1\left(s,X_s^{c}\right)\right|ds\\
\leq&\frac{1}{2}\int_0^s\Delta_t(b,c)^2ds+\frac{1}{2}||b_1-c_1||_T^2\int_0^t\left(1+\left|X_s^{c}\right|^{2q}\right)^2ds\\
\leq&\frac{1}{2}\int_0^s\Delta_t(b,c)^2ds+||b_1-c_1||_T^2\int_0^t\left(1+\left|X_s^{c}\right|^{4q}\right)ds\,.
\end{align*}
In the same way, we have
\begin{align*}
&-\int_0^t\left(X_s^{b}-X_s^{c}\right)\left(b_2\left(s,X_s^{b}\right)-c_2\left(s,X_s^{c}\right)\right)ds\\
\leq&\int_0^t\left|X_s^{b}-X_s^{c}\right|\left|b_2\left(s,X_s^{c}\right)-c_2\left(s,X_s^{c}\right)\right|ds\\
\leq&\frac{1}{2}\int_0^s\Delta_t(b,c)^2ds+\frac{1}{2}||b_2-c_2||_T^2\int_0^t\left(1+\left|X_s^{c}\right|^{2q}\right)^2ds\\
\leq&\frac{1}{2}\int_0^s\Delta_t(b,c)^2ds+||b_2-c_2||_T^2\int_0^t\left(1+\left|X_s^{c}\right|^{4q}\right)ds\,.
\end{align*}

We thus obtain:

\begin{equation*}
\EE\left[\Delta_t(b,c)^2\right]\leq2(\theta+1)\int_0^t\EE\left[\Delta_s(b,c)^2\right]ds+2\left(||b_1-c_1||_T^2+||b_2-c_2||_T^2\right)T\left(1+\widehat{\eta_{4q}^{c}}(T)\right)\,.
\end{equation*}

We apply Gr\"onwall lemma and we get:

\begin{equation*}
\EE\left[\Delta_t(b,c)^2\right]\leq2\left(||b_1-c_1||_T^2+||b_2-c_2||_T^2\right)T\left(1+\widehat{\eta_{4q}^{c}}(T)\right)e^{2(\theta+1)t}\,.
\end{equation*}
As the role of $b$ and $c$ can be inverted, we obtain:

\begin{equation*}
\EE\left[\Delta_t(b,c)^2\right]\leq2\left(||b_1-c_1||_T^2+||b_2-c_2||_T^2\right)T\left(1+\frac{1}{2}\widehat{\eta_{4q}^{b}}(T)+\frac{1}{2}\widehat{\eta_{4q}^{c}}(T)\right)e^{2(\theta+1)t}\,.
\end{equation*}
We combine this with Inequality \eqref{nadege} and we finally have \eqref{lala1} with the function 
\begin{equation*}
C_0'(x,y):=3\times2^{2q+\frac{3}{2}}C_q\sqrt{1+\frac{|x|+|y|}{2}}\sqrt{1+|x|^{\frac{4q-4}{4q}}+|y|^{\frac{4q-4}{4q}}}\,.
\end{equation*}

We obtain Inequalities \eqref{lala2}, \eqref{lala3} and \eqref{lala4} by proceeding similarly.
\end{proof}

As mentioned previously, we will use a fixed point theorem. We already have a continuous map. We will now restrict the space so that the map is a contraction.

\begin{defn}
\label{kaamelott}
Set $K>0$ and $T>0$. We consider
\begin{equation*}
\Lambda_T^K:=\left\{b\in\Lambda_T\,\,:\,\,||b||_T\leq K\right\}\,.
\end{equation*}
We also define $F_T^K:=\Lambda_T^K\times\Lambda_T^K\times\Lambda_T^K\times\Lambda_T^K$.
\end{defn}

\begin{lem}
\label{perceval}
Let $X_0$ and $Y_0$ be two random variables  such that $\EE\left[X_0^{8q^2}\right]<\infty$ and $\EE\left[Y_0^{8q^2}\right]<\infty$. Then, there exist two positive parameter $K$ and $T_0$ such that for any $T<T_0$, we have the two following properties:
\begin{enumerate}
 \item $F_T^K$ is stable by $\Gamma$: $\Gamma F_T^K\subset F_T^K$.
 \item The Lipschitz norm of the restriction of $\Gamma$ on $F_T^K$ is less than $\frac{1}{2}$.
\end{enumerate}
\end{lem}
\begin{proof}
{\bf Step 1.} From \eqref{chaudron}, we have
\begin{equation*}
\left|\left|p_1\circ\Gamma(b)\right|\right|_T\leq aC\left(1+\widehat{\eta_{2q}^{b}}(T)\right)\,.
\end{equation*}
So, from Lemma \ref{oumaima}, we have
\begin{align*}
\left|\left|p_1\circ\Gamma(b)\right|\right|_T\leq&aC\left(1+k_1(q)\left[T^{2q}+\left(||b_1-\rho_0||_T^{2q}+||b_2-\rho_0||_T^{2q}\right)\left(T^{2q}+\widehat{\eta_{8q^2}^{\rho}}(T)\right)\right]\right)\\
\leq&aC\left(1+k_1(q)\left[T^{2q}+2^{2q}\left(||b_1||_T^{2q}+||b_2||_T^{2q}+2||\rho_0||_T^{2q}\right)\left(T^{2q}+\widehat{\eta_{8q^2}^{\rho}}(T)\right)\right]\right)\\
\leq&aC\left(1+k_1(q)\left[T^{2q}+2^{2q+1}\left(K^{2q}+||\rho_0||_T^{2q}\right)\left(T^{2q}+\widehat{\eta_{8q^2}^{\rho}}(T)\right)\right]\right)\\
\leq&C\left(1+k_1(q)\left[T^{2q}+2^{2q+1}\left(K^{2q}+||\rho_0||_T^{2q}\right)\left(T^{2q}+\widehat{\eta_{8q^2}^{\rho}}(T)\right)\right]\right)\,.
\end{align*}
This can be rewritten as
\begin{equation*}
\left|\left|p_1\circ\Gamma(b)\right|\right|_T\leq C_1+C_2T^{2q}\left(1+K^{2q}\right)\,,
\end{equation*}
where $C_1$ and $C_2$ do not depend on $T$ nor on $K$. We take $K\geq 2C_1$ and $T_0\leq\left(\frac{C_1}{C_2(1+K^{2q})}\right)^{\frac{1}{2q}}$. As a consequence, for any $T<T_0$, we have
$$
\left|\left|p_1\circ\Gamma(b)\right|\right|_T\leq 2C_1\leq K,
$$ 
which proves that $p_1\circ\Gamma(b)\in\Lambda_T^K$. We proceed similarly and we obtain that $p_2\circ\Gamma(b)\in\Lambda_T^K$, $p_3\circ\Gamma(b)\in\Lambda_T^K$ and $p_4\circ\Gamma(b)\in\Lambda_T^K$. Consequently, $\Gamma(b)\in F_T^K$ if $b\in F_T^K$.\\[4pt]
{\bf Step 2.} We will now examine the Lipschitz constant. By making the sum of the inequalities in Lemma \ref{vin2}, we obtain
\begin{equation*}
\left|\left|\Gamma\left(b\right)-\Gamma\left(c\right)\right|\right|_T^F\leq\alpha(T)||b-c||_T^F\,,
\end{equation*}
with $\alpha(T):=\max\left\{2\sqrt{T}C_0'\left(\widehat{\eta_{4q}^{b}}(T),\widehat{\eta_{4q}^{c}}(T)\right);2\sqrt{T}C_0'\left(\widehat{\xi_{4q}^{b}}(T),\widehat{\xi_{4q}^{c}}(T)\right)\right\}$. We choose $T_2$ sufficiently small such that $\alpha(T_2)\leq\frac{1}{2}$. By taking $T:=\min\{T_1;T_2\}$, the Lipschitz norm is less than~$\frac{1}{2}$.
\end{proof}

We point out that $T$ depends on $X_0$ and $Y_0$. This is why we will only be able to construct, in a first time, a solution on a finite time interval.

\begin{prop}
\label{pyramide}
Let $X_0$ and $Y_0$ be two random variables such that $\EE\left[X_0^{8q^2}\right]<\infty$ and $\EE\left[Y_0^{8q^2}\right]<\infty$. Then there exists $T_0>0$ such that for any $T<T_0$, the system of equations \eqref{eq: SDEs} admits a strong solution on the interval $[0;T]$. Moreover, we have
\begin{equation*}
\sup_{0\leq t\leq T}\EE\left\{\left|X_t\right|^{4q}\right\}+\sup_{0\leq t\leq T}\EE\left\{\left|Y_t\right|^{4q}\right\}<\infty\,.
\end{equation*}
\end{prop}
\begin{proof}
{\bf Step 1.} We take $K$ and $T_0$ as defined in Lemma \ref{perceval}. Thus, the Lipschitz norm of the restriction of $\Gamma$ on $F_T^K$ is smaller than $\frac{1}{2}$ for any $T<T_0$.\\[2pt]
We take $b\in F_T^K$. We consider the sequence $\left(b_p\right)_{p\in\mathbb{N}}$ by $b_0:=b$ and $b_{p+1}:=\Gamma\left(b_p\right)$ for any $p\in\mathbb{N}$. We know that $b_p\in F_T^K$ for any $p\in\mathbb{N}$. $\Gamma$ being a contraction, the sequence $\left(b_p\right)_p$ converges to an element $b_\infty\in F_T^K$. This element does not depend on $b$. Moreover, we have $\Gamma\left(b_\infty\right)=b_\infty$. Consequently, $\left(X_t^{b_\infty},Y_t^{b_\infty}\right)_{t\in[0;T]}$ is a strong solution of the system \eqref{eq: SDEs} providing that $b_{\infty,1}$, $b_{\infty,2}$, $b_{\infty,3}$ and $b_{\infty,4}$ are locally Lipschitz (with $(b_{\infty,1},b_{\infty,2},b_{\infty,3},b_{\infty,4})=:b_\infty)$.\\[2pt]
As $b_{n+1}=\Gamma\left(b_n\right)$, then for $|x|\leq N$ and $|y|\leq N$, we have:
\begin{align*}
\left|b_{n+1,1}(t,x)-b_{n+1,1}(t,y)\right|=&a\left|\EE\left[F_{11}'\left(x-X_t^{b_n}\right)-F_{11}'\left(y-X_t^{b_n}\right)\right]\right|\\
\leq&a\EE\left[\left|F_{11}'\left(x-X_t^{b_n}\right)-F_{11}'\left(y-X_t^{b_n}\right)\right|\right]\\
\leq&2^{2q-1}a|x-y|\EE\left[c+|x|^{2q-1}+|y|^{2q-1}+2\left|X_t^{b_n}\right|\right]\\
\leq&k(N)|x-y|\left(1+\widehat{\eta_{2q-1}^{b_n}}(T)\right)\,.
\end{align*}
Since $||b_{n,1}||_T\leq K$ and $||b_{n,2}||_T\leq K$, from Lemma \ref{oumaima}, we deduce:
\begin{equation*}
\left|b_{n+1,1}(t,x)-b_{n+1,1}(t,y)\right|\leq \psi\left(N,K,T,\rho\right)|x-y|\,.
\end{equation*}
By taking the limit as $n$ goes to infinity, we deduce that 
\begin{equation*}
\left|b_{\infty,1}(t,x)-b_{\infty,1}(t,y)\right|\leq \psi\left(N,K,T,\rho\right)|x-y|\,.
\end{equation*}
So $b_{\infty,1}$ is locally Lipschitz. We can do the same reasoning for $b_{\infty,2}$, $b_{\infty,3}$ and $b_{\infty,4}$. Therefore, $\left(X_t^{b_\infty},Y_t^{b_\infty}\right)_{t\in[0;T]}$ is a strong solution of \eqref{eq: SDEs}.\\[4pt]
{\bf Step 2.} According to the assumptions, $\EE\left[X_0^{8q^2}\right]$ and $\EE\left[Y_0^{8q^2}\right]$ are finite. We thus deduce
\begin{equation*}
\sup_{s\in[0;T]}\EE\left[\left|X_s^{\rho}\right|^{8q^2}\right]+\sup_{s\in[0;T]}\EE\left[\left|Y_s^{\rho}\right|^{8q^2}\right]<\infty\,.
\end{equation*}
From Lemma \ref{oumaima}, we have:
\begin{equation*}
\widehat{\eta_{2n}^{b}}(T)\leq k_1(n)\left[T^{2n}+\left(||b_1-\rho_0||_T^{2n}+||b_2-\rho_0||_T^{2n}\right)\left(T^{2n}+\widehat{\eta_{4qn}^{\rho}}(T)\right)\right]\,,
\end{equation*}
and
\begin{equation*}
\widehat{\xi_{2n}^{b}}(T)\leq k_2(n)\left[T^{2n}+\left(||b_3-\rho_0||_T^{2n}+||b_4-\rho_0||_T^{2n}\right)\left(T^{2n}+\widehat{\xi_{4qn}^{\rho}}(T)\right)\right]\,.
\end{equation*}
As $\widehat{\eta_{8q^2}^{\rho}}(T)$ and $\widehat{\xi_{4qn}^{\rho}}(T)$ are finite, we deduce the finiteness of $\widehat{\eta_{2n}^{b_\infty}}(T)$ and of $\widehat{\xi_{2n}^{b_\infty}}(T)$ for any $n$ such that $4qn\leq8q^2$. Consequently, we have $\widehat{\eta_{4q}^{b_\infty}}(T)+\widehat{\xi_{4q}^{b_\infty}}(T)<\infty$.
\end{proof}

Let us point out that the uniqueness of the solution has not been proved for the moment. It will be proved subsequently.

We just obtained the result in finite time. We aim to establish it on the whole set $\bRb_+$. To this end, we will assume that there exists a maximal time such that after this time, there is explosion. We will give a good control of the moments and then extend the solution after the maximal time. Thus we will obtain a contradiction proving that there is no such maximal time.

\begin{lem}
\label{pomme}
Let $X_0$ and $Y_0$ be two random variables  such that $\EE\left[X_0^{2k}\right]<\infty$ and $\EE\left[Y_0^{2k}\right]<\infty$ with $k>q$. Let $T$ be a positive real and $b$ an element of $F_T$. We assume that the function $\Gamma\left(b\right)$ is defined (which is not obvious since we did not assume the finiteness of the $8q^2$-th moment) and that it satisfies $\Gamma\left(b\right)=b$. We put $X_t:=X_t^{b}$ and $Y_t:=Y_t^{b}$ the strong solution of the system starting from $X_0$ with the drift defined by $b$. Then, there exists a function $C''$ such that
\begin{equation*}
\widehat{x_T}:=\sup_{t\in[0;T]}\EE\left\{\left|X_t\right|^{2k}\right\}\leq C''\left(\EE\left[\left|X_0\right|^{2k}\right];\EE\left[\left|Y_0\right|^{2k}\right]\right)\,,
\end{equation*}
and
\begin{equation*}
\widehat{y_T}:=\sup_{t\in[0;T]}\EE\left\{\left|Y_t\right|^{2k}\right\}\leq C''\left(\EE\left[\left|X_0\right|^{2k}\right];\EE\left[\left|Y_0\right|^{2k}\right]\right)\,.
\end{equation*}

\end{lem}
\begin{proof}
{\bf Step 1.} We put $x_t:=\EE\left\{\left|X_t\right|^{2k}\right\}$ and $y_t:=\EE\left\{\left|Y_t\right|^{2k}\right\}$. We apply It\^o formula, we take the integration, the expectation then we derive:

\begin{align*}
\frac{d}{dt}x_t=&-2k\EE\left[{\rm sign}\left(X_t\right)\left|X_t\right|^{2k-1}\left(V_1'\left(X_t\right)+F_{11}'\ast\mu_t\left(X_t\right)+F_{12}'\ast\nu_t\left(X_t\right)\right)\right]\\
&+k(2k-1)\sigma^2\EE\left[\left|X_t\right|^{2k-2}\right]\,.
\end{align*}
Since $F_{11}$ is convex, it is easy to prove that $\EE\left[{\rm sign}\left(X_t\right)\left|X_t\right|^{2k-1}F_{11}'\ast\mu_t\left(X_t\right)\right]\geq0$. We deduce
\begin{align*}
\frac{d}{dt}x_t\leq&-2k\EE\left[{\rm sign}\left(X_t\right)\left|X_t\right|^{2k-1}\left(V_1'\left(X_t\right)+F_{12}'\ast\nu_t\left(X_t\right)\right)\right]\\
&+k(2k-1)\sigma^2x_t^{1-\frac{1}{k}}\,.
\end{align*}
{\bf Step 2.} We now prove that the right hand side of the inequality is negative if $x_t$ and $y_t$ are too large. If $V_1$ was convex, the integral term with $V'$ would be easy to control. However, $V$ is not convex. We take $\tau>0$ arbitrarily large and we have:

\begin{align*}
\EE\left[{\rm sign}\left(X_t\right)\left|X_t\right|^{2k-1}V_1'\left(X_t\right)\right]=&\int_{-\tau}^\tau{\rm sign}(x)|x|^{2k-1}V_1'(x)\mu_t(dx)\\
&+\int_{[-\tau;\tau]^c}{\rm sign}(x)|x|^{2k-1}V_1'(x)\mu_t(dx)\,.
\end{align*}
If $\tau$ is large enough, the second integral is positive whilst the first can be negative.\\[2pt]
{\bf Step 2.1.} We begin by the first integral:
\begin{align*}
&\int_{-\tau}^\tau{\rm sign}(x)|x|^{2k-1}V_1'(x)\mu_t(dx)\geq-\int_{-\tau}^\tau|x|^{2k-1}\left|V_1'(x)\right|\mu_t(dx)\\
&\geq-|\tau|^{2k-1}\sup_{x\in[-\tau;\tau]}\left|V_1'(x)\right|=:-f(\tau)\,.
\end{align*}
{\bf Step 2.2.} We now look at the second integral. Since $V''(\pm\infty)=+\infty$, we know that $g(\tau):=\inf_{x\in[-\tau;\tau]^c}{\rm sign}(x)\frac{V_1'(x)}{|x|}>0$ if $\tau$ is large enough. Thus:
\begin{align*}
\int_{[-\tau;\tau]^c}{\rm sign}(x)|x|^{2k-1}V_1'(x)\mu_t(dx)\geq& g(\tau)\int_{[-\tau;\tau]^c}|x|^{2k}\mu_t(dx)\\
\geq& g(\tau)\left(\int_{\bRb}|x|^{2k}\mu_t(dx)-\int_{[-\tau;\tau]}|x|^{2k}\mu_t(dx)\right)\\
\geq& g(\tau)\left(x_t-\tau^{2k}\right)\,.
\end{align*}
{\bf Step 2.3.} We now control the mixed term $-2k\EE\left[{\rm sign}\left(X_t\right)\left|X_t\right|^{2k-1}F_{12}'\ast\nu_t\left(X_t\right)\right]$. We take $\widetilde{Y_t}$ an independent copy of $Y_t$. Then, we have:
\begin{align*}
-2k\EE\left[{\rm sign}\left(X_t\right)\left|X_t\right|^{2k-1}F_{12}'\ast\nu_t\left(X_t\right)\right]=&-2k\EE\left[{\rm sign}\left(X_t\right)\left|X_t\right|^{2k-1}F_{12}'\left(X_t-\widetilde{Y_t}\right)\right]\\
\leq&2k\EE\left[\left|X_t\right|^{2k-1}\left|F_{12}'\left(X_t-\widetilde{Y_t}\right)\right|\right]\,.
\end{align*}
Since $F_{12}'(0)=0$ and $F_{12}'$ is Lipschitz, we obtain:
\begin{align*}
-2k\EE\left[{\rm sign}\left(X_t\right)\left|X_t\right|^{2k-1}F_{12}'\ast\nu_t\left(X_t\right)\right]\leq&2kC\EE\left[\left|X_t\right|^{2k-1}\left|X_t-\widetilde{Y_t}\right|\right]\\
\leq&2kC\left\{\EE\left[\left|X_t\right|^{2k}\right]+\EE\left[\left|X_t\right|^{2k-1}\right]\EE\left[\left|Y_t\right|\right]\right\}\\
\leq&2kCx_t+2kCx_t^{1-\frac{1}{2k}}y_t^{\frac{1}{2k}}\,.
\end{align*}
{\bf Step 3.} We combine the inequalities and we get:
\begin{align*}
\frac{d}{dt}x_t\leq&k(2k-1)\sigma^2x_t^{1-\frac{1}{k}}+2kf(\tau)-2kg(\tau)x_t\\
&+2kg(\tau)\tau^{2k}+2kCx_t+2kCx_t^{1-\frac{1}{2k}}y_t^{\frac{1}{2k}}\,.
\end{align*}

In the same way, we have:
\begin{align*}
\frac{d}{dt}y_t\leq&k(2k-1)\sigma^2y_t^{1-\frac{1}{k}}+2kf(\tau)-2kg(\tau)y_t\\
&+2kg(\tau)\tau^{2k}+2kCy_t+2kCy_t^{1-\frac{1}{2k}}x_t^{\frac{1}{2k}}\,.
\end{align*}

If $y_t\leq x_t$, we thus have
\begin{align*}
\frac{d}{dt}x_t\leq&k(2k-1)\sigma^2x_t^{1-\frac{1}{k}}+2kf(\tau)-2kg(\tau)x_t\\
&+2kg(\tau)\tau^{2k}+4kCx_t\,.
\end{align*}
By taking $\tau$ large enough, $g(\tau)>4C$ so that if $x_t$ is larger than a constant $\chi(\tau)$, $\frac{d}{dt}x_t\leq0$.\\[2pt]
Conversely, if $x_t\leq y_t$, if $y_t$ is larger than $\chi(\tau)$, $\frac{d}{dt}y_t\leq0$. We immediately deduce:
\begin{equation*}
\widehat{x_T}\leq\max\left\{\chi(\tau);\EE\left[\left|X_0\right|^{2k}\right];\EE\left[\left|Y_0\right|^{2k}\right]\right\}
\end{equation*}
and
\begin{equation*}
\widehat{y_T}\leq\max\left\{\chi(\tau);\EE\left[\left|X_0\right|^{2k}\right];\EE\left[\left|Y_0\right|^{2k}\right]\right\}
\end{equation*}

\end{proof}

We are now able to obtain the main theorem

\begin{thm}
\label{netflix}
Set two random variables $X_0$ and $Y_0$ such that $\EE\left[X_0^{2q}\right]<\infty$ and $\EE\left[Y_0^{2q}\right]<\infty$. Then, the system admits a unique strong solution on $\bRb_+$.
\end{thm}
\begin{proof}
{\bf Step 1.} We consider 
\begin{equation*}
U:=\sup\left\{T>0\,\,:\,\,\mbox{(E) admits a unique solution on }[0;T], \sup_{0\leq t\leq T}\EE\left[X_t^{8q^2}\right]+\sup_{0\leq t\leq T}\EE\left[X_t^{8q^2}\right]<\infty\right\}
\end{equation*}

with the convention $\sup\emptyset=0$. We begin to show that $U>0$. By taking $K$ large enough, there exists $T>0$ and a unique $b\in F_T^K$ such that $\Gamma\left(b\right)=b$. Then $\left(X^{b},Y^{b}\right)$ is a strong solution of the system (E) on $[0;T]$. We now consider a solution $(\widetilde{X_t},\widetilde{Y_t})_{t\in[0;T]}$. To this solution, we associate the following drifts
\begin{align*}
&c_1(t,x):=a\EE\left[F_{11}'(x-\widetilde{X_t})\right]\,\,,c_2(t,x):=(1-a)\EE\left[F_{12}'(x-\widetilde{Y_t})\right]\,,\\
&c_3(t,x):=a\EE\left[F_{21}'(x-\widetilde{X_t})\right]\quad\mbox{and}\quad c_4(t,x):=(1-a)\EE\left[F_{22}'(x-\widetilde{Y_t})\right]\,.
\end{align*}
We put $c:=(c_1,c_2,c_3,c_4)$. By the assumptions on $F_{11}$, we obtain:
\begin{equation*}
\frac{\left|c_1(t,x)\right|}{1+x^{2q}}=a\frac{\left|F_{11}'\left(x-\widetilde{X_t}\right)\right|}{1+x^{2q}}\leq C\left(1+\sup_{0\leq t\leq T}\EE\left[\left|\widetilde{X_t}\right|^{2q}\right]\right)\,.
\end{equation*}
In the same way, we have 
\begin{align*}
&\frac{\left|c_2(t,x)\right|}{1+x^{2q}}\leq C\left(1+\sup_{0\leq t\leq T}\EE\left[\left|\widetilde{Y_t}\right|^{2q}\right]\right)\,,\\
&\frac{\left|c_3(t,x)\right|}{1+x^{2q}}\leq C\left(1+\sup_{0\leq t\leq T}\EE\left[\left|\widetilde{X_t}\right|^{2q}\right]\right)\,,\\
\mbox{and}\quad&\frac{\left|c_4(t,x)\right|}{1+x^{2q}}\leq C\left(1+\sup_{0\leq t\leq T}\EE\left[\left|\widetilde{Y_t}\right|^{2q}\right]\right)\,.
\end{align*}
However, according to Lemma \ref{pomme}, the moment at time $t$ of $\widetilde{X_t}$ and the one of $\widetilde{Y_t}$ are bounded by a function which depends on the moments of $X_0$ and $Y_0$.\\[2pt]
If we have a function $c$ such that $\Gamma(c)=c$ then the associated process verifies the equations of Lemma \ref{pomme}. Then, by taking $K$ large enough, we deduce that $\frac{\left|c_1(t,x)\right|}{1+x^{2q}}\leq K$, $\frac{\left|c_2(t,x)\right|}{1+x^{2q}}\leq K$, $\frac{\left|c_3(t,x)\right|}{1+x^{2q}}\leq K$ and $\frac{\left|c_4(t,x)\right|}{1+x^{2q}}\leq K$. This means that $c\in F_T^K$. Consequently, for any random variables $X_0$ and $Y_0$ with $8q^2$th moment finite, by taking $K$ large enough (which depends on the initial moments), we know that a solution of the equation $\Gamma(c)=c$ is in $F_T^K$. However, the equation has a unique solution on $F_T^K$ since the map $\Gamma$ is a contraction on $F_T^K$.\\[2pt]
We immediately deduce that there is a unique system of stochastic differential equations which corresponds to (E). And, there is a unique strong solution to this equation. We thus deduce that $U\geq T>0$.\\[4pt]
{\bf Step 2.} We assume that $U<\infty$. This $U$ does depend on the moments of $X_0$ and $Y_0$. We know that the moments of order $1$ to $8q^2$ of $(X_t,Y_t)$ are bounded by a constant $C_0''$ which depends only on the initial moments. We thus consider the system of equations
\begin{align*}
&X_t'=X_0'-\int_0^t\nabla V_1(X_s')ds-a\int_0^t(\nabla F_{11}\ast\mu_s)(X_s')ds-(1-a)\int_0^t(\nabla F_{12}\ast\nu_s)(X_s')ds+\sigma B_t\,,\\
&Y_t'=Y_0'-\int_0^t\nabla V_2(Y_s')ds-a\int_0^t(\nabla F_{21}\ast\mu_s)(Y_s')ds-(1-a)\int_0^t(\nabla F_{22}\ast\nu_s)(Y_s')ds+\sigma \widetilde{B_t}\,,
\end{align*}
with $X_0'$ and $Y_0'$ such that $\EE\left[\left|X_0'\right|^r\right]\leq C_0''$ and $\EE\left[\left|Y_0'\right|^r\right]\leq C_0''$. We can associate a time $T'>0$ to this equation such that it admits a unique strong solution on $[0;T']$. To the new random variable $X_0'$ is associated a new constant $K'$. Without any change to the generality, we take $K'\geq K$.\\[2pt]
We put $X_0':=X_{U-\frac{T'}{2}}$ and $Y_0':=Y_{U-\frac{T'}{2}}$. These new initial random variables satisfy the conditions so we can define a unique strong solution on $[0;T']$. This implies that we have extended $\left(X_t,Y_t\right)_{t\in[0;U]}$ to $\left(X_t,Y_t\right)_{t\in[0;U+\frac{T'}{2}]}$. Indeed, on $[U-\frac{T'}{2};U[$, there is uniqueness. This contradicts the definition of $U$.

\end{proof}

By using the proof of Theorem \ref{netflix}, we can directly obtain the following result:

\begin{prop}
\label{arthur}
Let $(X,Y)$ be a solution of the system \eqref{eq: SDEs}. Assume that $\EE\left(X_0^{2n}\right)+\EE\left(Y_0^{2n}\right)<\infty$ for some $n\in\mathbb{N}^*$. Then, we have
\begin{equation*}
\sup_{t\geq0}\EE\left(X_t^{2n}\right)+\sup_{t\geq0}\EE\left(Y_t^{2n}\right)<\infty\,.
\end{equation*}
\end{prop}

\section{Propagation of chaos}
\label{sec: PoC}
We recall the interacting particle system defined in \eqref{eq: many SDEs}:
\begin{subequations}
\begin{align*}
dX_t^{i}&=-\nabla V_1(X^{i}_t)\,dt-\frac{1}{N_n+M_n}\sum_{j=1}^{N_n}\nabla F_{11}(X^{i}_t-X^{j}_t)\,dt\notag
\\&\qquad-\frac{1}{N_n+M_n}\sum_{k=1}^{M_n}\nabla F_{12}(X^{i}_t-Y^{k}_t)\,dt+\sigma dW^i_t;~~i=1,\ldots, N_n;
\\dY^{i}_t&=-\nabla V_2(Y^{i}_t)\,dt-\frac{1}{N_n+M_n}\sum_{j=1}^{N_n}\nabla F_{21}(Y^{i}_t-X^{j}_t)\,dt\notag
\\&\qquad-\frac{1}{N_n+M_n}\sum_{k=1}^{M_n}\nabla F_{22}(Y^{i}_t-Y^{k}_t)\,dt+\sigma d\widetilde W^i_t,~~i=1,\ldots, M_n;
\end{align*}
\end{subequations}
and its identically independent copies given in \eqref{eq: MV-system}
\begin{equation*}
\begin{cases}
d\widehat{X_t^{i}}=-\nabla V_1(\widehat{X_t^{i}})\,dt-a (\nabla F_{11}\ast\mu_t)(\widehat{X_t^{i}})\,dt-(1-a)(\nabla F_{12}\ast\nu_t)(\widehat{X_t^{i}})\,dt+\sigma dW^i_t
\\ \hspace*{10cm} i=1,\ldots, N_n;
\\ d\widehat{Y_t^{i}}=-\nabla V_2(\widehat{Y_t^{i}})\,dt-a(\nabla F_{21}\ast\mu_t)(\widehat{Y_t^{i}})\,dt-(1-a)(\nabla F_{22}\ast\nu_t)(\widehat{Y_t^{i}})\,dt+\sigma d\widetilde W^i_t,
\\ \hspace*{10cm} i=1,\ldots, M_n,
\end{cases}
\end{equation*}
where $a=\lim\limits_{n\to\infty}\frac{N_n}{N_n+M_n}$ and $\mu_t=\mathrm{Law}(\widehat{X_t^{i}}), ~\nu_t=\mathrm{Law}(\widehat{Y_t^{i}})$.

The rest of this section is devoted to prove Theorem \ref{theo: PoC} that is to show that the interacting particle system satisfies propagation of chaos . We adapt the proof of \cite{Benachour98a,Herrmann02}. In Proposition~\ref{prop: 2nd} we prove a weaker statement than~\eqref{eq: E-sup} where the expectation and the supremum are interchanged. We then  strengthens Proposition~\ref{prop: 2nd} to the fourth power in Proposition \ref{prop: 4th}. Finally, Theorem \ref{theo: PoC} will be derived from these propositions. 

We will need the following lemma on a nonlinear generalisation of Gr\"onwall’s inequality. 
\begin{lemma}
\label{lem: Gronwall}
Let $\phi$  be a positive function such that $\phi(0)=0$. Suppose that there exist constants $A>0, B\geq 0$ and $0\leq\alpha<1$ such that
\[
\phi(t)\leq A\int_0^t \phi(s)\,ds+B\int_0^t\phi(s)^\alpha\,ds,
\]
then
\[
\phi(t)\leq \Big(\frac{B}{A}\big(e^{(1-\alpha)A t}-1\big)\Big)^{\frac{1}{1-\alpha}}.
\]
\end{lemma}
\begin{proof}
We note that a special case of this lemma for $\alpha=\frac{1}{2}$ has appeared in \cite[Lemma 2]{Herrmann02} while a more general version where $A$ and $B$ are functions of $s$ can be found in \cite[Theorem 21]{Dragomir03book}. For the convenience of the reader we provide a simplified proof for the case of constant coefficients and arbitrary $\alpha$ here. Suppose $\psi$ solve the integral equation
\[
\psi(t)=A \int_0^t \psi(s)\,ds+B\int_0^t\psi(s)^\alpha\,ds\quad \psi(0)=0.
\]
We take the derivative with respect to $t$ both sides to obtain
\[
\frac{d\psi(t)}{A\psi(t)+B\psi(t)^\alpha}=dt,\quad\psi(0)=0.
\] 
Taking the anti-derivative of this ODE gives
\[
\frac{1}{A(\alpha-1)}\log\bigg[\frac{\psi(t)^\alpha}{A\psi(t)+B\psi(t)^\alpha}\bigg]=t+C.
\] 
Solving this equation with the initial data $\psi(0)=0$ we obtain
\[
\psi(t)=\bigg(\frac{B}{A}\Big(e^{A(1-\alpha)t}-1\Big)\bigg)^\frac{1}{1-\alpha}.
\]
A comparison principle gives $\phi(t)\leq \psi(t)$, which is the assertion of the lemma. 
\end{proof}
%

\begin{proposition}
\label{prop: 2nd}
We have
\begin{equation}
\label{eq: 2nd}
\lim\limits_{n\to\infty}\sup\limits_{t\in[0,T]}\E\Big[\big(X_t^{i}-\widehat{X_t^{i}}\big)^2\Big]=0\quad\text{and}\quad \lim\limits_{n\to\infty}\sup\limits_{t\in[0,T]}\E\Big[\big(Y_t^{i}-\widehat{Y_t^{i}}\big)^2\Big]=0.
\end{equation}
\end{proposition}
\begin{proof}
We define
\begin{equation}
\omega(t):=\E\left[\big(X_t^{1}-\widehat{X_t^{1}}\big)^2\right]\quad\text{and}\quad\widehat{\omega}(t):=\E\left[\big(Y_t^{1}-\widehat{Y_t^{1}}\big)^2\right]\,.
\end{equation}
We have
\begin{equation}
\omega(t)=\E\left[\big(X_t^{i}-\widehat{X_t^{i}}\big)^2\right]~~\forall i=1,\ldots, N_n \quad\text{and}\quad\widehat{\omega}(t)=\E\left[\big(Y_t^{i}-\widehat{Y_t^{i}}\big)^2\right]~~\forall i=1,\ldots, M_n.
\end{equation}
Using the It\^{o} formula, we compute
\begin{align}
\label{eq: est1}
\omega(t)&=\E\left[\big(X_t^{i}-\widehat{X_t^{i}}\big)^2\right]\notag
\\&=-2\E\int_0^t (\nabla V_1(X_s^{i})-\nabla V_1(\widehat{X_s^{i}}))\cdot (X_s^{i}-\widehat{X_s^{i}}) \,ds\notag
\\&\qquad-2\E\int_0^t \Big[\frac{1}{N_n+M_n}\sum_{j=1}^{N_n}\nabla F_{11}(X_s^{i}-X_s^{j})-a(\nabla F_{11}\ast\mu_s)(\widehat{X_s^{i}})\Big]\cdot (X_s^{i}-\widehat{X_s^{i}}) \,ds\notag
\\&\qquad-2\E\int_0^t \Big[\frac{1}{N_n+M_n}\sum_{k=1}^{M_n}\nabla F_{12}(X_s^{i}-Y_s^{k})-(1-a)(\nabla F_{12}\ast\mu_s)(\widehat{X_s^{i}})\Big]\cdot (X_s^{i}-\widehat{X_s^{i}}) \,ds\notag
\\&=2\int_0^t \E\big(A_i(s)+B_i(s)+C_i(s)+D_i(s)+E_i(s)\big)\,ds,
\end{align}
where
\begin{align*}
A_i(t)&=-\Big(\nabla V_1(X_t^{i})-\nabla V_1(\widehat{X_t^{i}})\Big)\cdot(X_t^{i}-\widehat{X_t^{i}}),\\
B_i(t)&=\bigg[-\frac{1}{N_n+M_n}\sum_{j=1}^{N_n}\Big(\nabla F_{11}(X_t^{i}-X_t^{j})-(\nabla F_{11}\ast\mu_t)(\widehat{X_t^{i}})\Big)\bigg]\cdot (X_{t}^{i}-\widehat{X_t^{i}}),
\\C_i(t)&=\Big(a-\frac{N_n}{N_n+M_n}\Big)(\nabla F_{11}\ast \mu_t)(\widehat{X_t^{i}})\cdot(X_t^{i}-\widehat{X_t^{i}}),
\\D_i(t)&=\bigg[-\frac{1}{N_n+M_n}\sum_{k=1}^{M_n}\Big(\nabla F_{12}(X_t^{i}-Y_t^{k})-(\nabla F_{12}\ast\nu_t)(\widehat{X_t^{i}})\Big)\bigg]\cdot (X_{t}^{i}-\widehat{X_t^{i}}),
\\E_i(t)&=\Big((1-a)-\frac{M_n}{N_n+M_n}\Big)(\nabla F_{12}\ast\nu_t)(\widehat{X_t^{i}})\cdot(X_t^{i}-\widehat{X_t^{i}}).
\end{align*}
Next we estimate each term in \eqref{eq: est1}. We start with $A_i$:
\begin{align*}
A_i(t)&=-\Big(\nabla V_1(X_t^{i})-\nabla V_1(\widehat{X_t^{i}})\Big)\cdot(X_t^{i}-\widehat{X_t^{i}})
\\&\overset{\eqref{eq: V1}}{\leq}\theta_1\Big(X_t^{i}-\widehat{X_t^{i}}\Big)^2.
\end{align*} 
This implies that 
\begin{equation}
\label{eq: Ai}
\sum_{i=1}^{N_n}\E\left[A_i(t)\right]\leq N_n\theta_1 \omega(t).
\end{equation}
Next we estimate $B_i$:
\begin{align}
\label{eq: Bi}
B_i(t)&=\bigg[-\frac{1}{N_n+M_n}\sum_{j=1}^{N_n}\Big(\nabla F_{11}(X_t^{i}-X_t^{j})-(\nabla F_{11}\ast\mu_t)(\widehat{X_t^{i}})\Big)\bigg]\cdot (X_{t}^{i}-\widehat{X_t^{i}})\notag
\\&=\bigg[-\frac{1}{N_n+M_n}\sum_{j=1}^{N_n}\Big(\nabla F_{11}(X_t^{i}-X_t^{j})-\nabla F_{11}(\widehat{X_t^{i}}-\widehat{X_t^{j}})\Big)\bigg]\cdot (X_{t}^{i}-\widehat{X_t^{i}})\notag
\\&\qquad-\bigg[\frac{1}{N_n+M_n}\sum_{j=1}^{N_n}\Big(\nabla F_{11}(\widehat{X_t^{i}}-\widehat{X_t^{j}})-(\nabla F_{11}\ast\mu_t)(\widehat{X_t^{i}})\Big)\bigg]\cdot (X_{t}^{i}-\widehat{X_t^{i}})\notag
\\&=-\frac{1}{N_n+M_n}\sum_{j=1}^{N_n}\varrho^1_{ij}(t)-\frac{1}{N_n+M_n}\sum_{j=1}^{N_n}\varrho^2_{ij}(t),
\end{align}
where
\begin{subequations}
\label{eq: rhoij}
\begin{align}
\varrho^1_{ij}(t)&:=\Big(\nabla F_{11}(X_t^{i}-X_t^{j})-\nabla F_{11}(\widehat{X_t^{i}}-\widehat{X_t^{j}})\Big)\cdot (X_{t}^{i}-\widehat{X_t^{i}}),
\\\varrho^2_{ij}(t)&:=\Big(\nabla F_{11}(\widehat{X_t^{i}}-\widehat{X_t^{j}})-(\nabla F_{11}\ast\mu_t)(\widehat{X_t^{i}})\Big)\cdot (X_{t}^{i}-\widehat{X_t^{i}}).
\end{align}
\end{subequations}
We have
\begin{align*}
\sum_{i=1}^{N_n}\sum_{j=1}^{N_n}\varrho^{1}_{ij}(t)=\sum_{1\leq i<j\leq N_n}\varrho^3_{ij}(t),
\end{align*}
where $\varrho^3_{ij}(t)=\varrho^1_{ij}(t)+\varrho^1_{ji}(t)$. Since $\nabla F_{11}$ is an odd function, we have
\begin{align*}
\varrho^3_{ij}(t)&=\Big(\nabla F_{11}(X_t^{i}-X_t^{j})-\nabla F_{11}(\widehat{X_t^{i}}-\widehat{X_t^{j}})\Big)\cdot (X_{t}^{i}-\widehat{X_t^{i}})
\\&\qquad+\Big(\nabla F_{11}(X_t^{j}-X_t^{i})-\nabla F_{11}(\widehat{X_t^{j}}-\widehat{X_t^{i}})\Big)\cdot (X_{t}^{j}-\widehat{X_t^{j}})
\\&=\Big(\nabla F_{11}(X_t^{i}-X_t^{j})-\nabla F_{11}(\widehat{X_t^{i}}-\widehat{X_t^{j}})\Big)\cdot \Big((X_{t}^{i}-\widehat{X_t^{i}})-(X_{t}^{j}-\widehat{X_t^{j}})\Big).
\end{align*}
If $X_t^{i}-X_t^{j}\geq \widehat{X_t^{i}}-\widehat{X_t^{j}}$ (resp. $X_t^{i}-X_t^{j}\leq \widehat{X_t^{i}}-\widehat{X_t^{j}}$) then $X_t^{i}-\widehat{X_t^{i}}\geq X_t^{j}-\widehat{X_t^{j}}$ (resp. $X_t^{i}-\widehat{X_t^{i}}\leq X_t^{j}-\widehat{X_t^{j}}$) and $\nabla F_{11}(X_t^{i}-X_t^{j})\geq \nabla F_{11}(\widehat{X_t^{i}}-\widehat{X_t^{j}})$ (resp. $\nabla F_{11}(X_t^{i}-X_t^{j})\leq \nabla F_{11}(\widehat{X_t^{i}}-\widehat{X_t^{j}})$ ) as $\nabla F_{11}$ is increasing. Thus we always have $\varrho^3_{ij}(t)\geq 0$. Therefore, 
\begin{equation}
\label{eq: varrho1}
\sum_{i,j=1}^{N_n}\varrho^1_{ij}(t)\geq 0.
\end{equation}
On the other hand, using Cauchy-Schwarz inequality, we get
\begin{equation}
\label{eq:sumvarrho2}
\E\Big(\sum_{j=1}^{N_n}\varrho^2_{ij}(t)\Big)\leq \Big(\E\big((X_{t}^{i}-\widehat{X_t^{i}})^2 \big)\kappa_i(t)\Big)^{\frac{1}{2}},
\end{equation}
where
\begin{align*}
\kappa_i(t)=\E\Bigg(\Big[\sum_{j=1}^{N_n}\big(\nabla F_{11}(\widehat{X_t^{i}}-\widehat{X_t^{j}})-(\nabla F_{11}\ast\mu_t)(\widehat{X_t^{i}})\big)\Big]^2\Bigg).
\end{align*}
We rewrite $\kappa_i$ as 
\begin{align*}
&\kappa_i(t)=\sum_{j=1}^{N_n}\xi_{j,j}(t)+\sum_{1\leq j<k\leq N_n}\xi_{j,k}(t)\quad\text{with}
\\& \xi_{j,k}(t)=\E\Big(\big[\nabla F_{11}(\widehat{X_t^{i}}-\widehat{X_t^{j}})-(\nabla F_{11}\ast\mu_t)(\widehat{X_t^{i}})\big]\big[\nabla F_{11}(\widehat{X_t^{i}}-\widehat{X_t^{k}})-(\nabla F_{11}\ast\mu_t)(\widehat{X_t^{i}})\big]\Big).
\end{align*}
If $j\neq k$, $\widehat{X^{i}}, \widehat{X^{j}}$ and $\widehat{X^{k}}$ are three independent copies of $\widehat{X^{1}}$. This implies that
\begin{align*}
\xi_{j,k}&=\E_{\widehat{X^{i}}}\Big(\E_{\widehat{X^{j}}}\big[\nabla F_{11}(\widehat{X_t^{i}}-\widehat{X_t^{j}})-(\nabla F_{11}\ast\mu_t)(\widehat{X_t^{i}})\big]\E_{\widehat{X^{k}}}\big[\nabla F_{11}(\widehat{X_t^{i}}-\widehat{X_t^{k}})-(\nabla F_{11}\ast\mu_t)(\widehat{X_t^{i}})\big]\Big)
\\&=\E_{\widehat{X^{i}}}[0]=0,
\end{align*}
where we have used the fact that $\widehat{X^{i}}, \widehat{X^{j}}$ and $\widehat{X^{k}}$ have the same law $\mu_t$.
 For $j=k$, we get
\begin{align}
\xi_{j,j}(t)&=\E\Big(\big[\nabla F_{11}(\widehat{X_t^{i}}-\widehat{X_t^{j}})-(\nabla F_{11}\ast\mu_t)(\widehat{X_t^{i}})\big]^2\Big)\notag
\\&\leq 2\E\Big(|\nabla F_{11}(\widehat{X_t^{i}}-\widehat{X_t^{j}})|^2+|(\nabla F_{11}\ast\mu_t)(\widehat{X_t^{i}})|^2\Big).
\end{align}
Since $|\nabla F_{11}(x)|\leq C(1+|x|^{2q})$, we have $|\nabla F_{11}(x)|^2\leq C(1+|x|^{4q})$. Applying this inequality we obtain 
\begin{align}
\label{eq: bound convol term}
&\E|\nabla F_{11}(\widehat{X_t^{i}}-\widehat{X_t^{j}})|^2\leq C\E(1+|\widehat{X_t^{i}}-\widehat{X_t^{j}}|^{4q})\leq C\E(1+|\widehat{X_t^{i}}|^{4q}+|\widehat{X_t^{j}}|^{4q})\leq C,\notag
\\& \E(|(\nabla F_{11}\ast\mu_t)(\widehat{X_t^{i}})|^2)=\E\Big(\Big|\int \nabla F_{11}(\widehat{X_t^{i}}-y)\mu_t(y)\,dy\Big|^2\Big)\leq \E\Big(\int |\nabla F_{11}(\widehat{X_t^{i}}-y)|^2\mu_t(dy)\Big)\notag
\\&\hspace*{3.5cm}\leq C \E\Big(\int (1+|\widehat{X_t^{i}}|^{4q}+|y|^{4q})\mu_t(y)\,dy\Big)\leq C\E\Big(1+|\widehat{X_t^{i}}|^{4q}\Big)\leq C.
\end{align}
Therefore, we obtain 
$$
\kappa_i(t)=\sum_{j=1}^{N_n}\xi_{j,j}(t)\leq CN_n
$$
Substituting this estimate back into \eqref{eq:sumvarrho2} gives 
\begin{equation}
\E\Big(\sum_{j=1}^{N_n}\rho^2_{ij}(t)\Big)\leq CN_n^{\frac{1}{2}}\Big(\E(X_t^{i}-\widehat{X_t^{i}})^2\Big)^{\frac{1}{2}}.
\label{eq: rho2ij}
\end{equation}
Substituting \eqref{eq: rho2ij} and \eqref{eq: varrho1} back into \eqref{eq: Bi}, we achieve
\begin{equation}
\label{eq: Bi-2}
\sum_{i=1}^{N_n}\E\left[B_i(t)\right]\leq \frac{C N_n^{3/2}}{N_n+M_n}\Big(\E(X_t^{i}-\widehat{X_t^{i}})^2\Big)^{\frac{1}{2}}=\frac{C N_n^{3/2}}{N_n+M_n}\sqrt{\omega(t)}.
\end{equation}
We proceed with estimating $C_i$. Using Cauchy-Schwarz inequality we get
\begin{align*}
|\E\left[C_i(t)\right]|&=\Big|a-\frac{N_n}{N_n+M_n}\Big|\E (\nabla F_{11}\ast\mu_t)(\widehat{X_t^{i}})(X_t^{i}-\widehat{X_t^{i}})|
\\&\leq \Big|a-\frac{N_n}{N_n+M_n}\Big| \Big(\E\big[X_t^{i}-\widehat{X_t^{i}}\big]^2\Big)^{\frac{1}{2}}\Big(\E|(\nabla F_{11}\ast\mu_t)(\widehat{X_t^{i}})|^2\Big)^{\frac{1}{2}}
\\&\leq C \Big|a-\frac{N_n}{N_n+M_n}\Big|\Big(\E\big[X_t^{i}-\widehat{X_t^{i}}\big]^2\Big)^{\frac{1}{2}},
\end{align*}
where we have used $\E|(\nabla F_{11}\ast\mu_t)(\widehat{X_t^{i}})|^2\leq C$ which was proved in \eqref{eq: bound convol term}. Therefore,
\begin{equation}
\label{eq: C_i}
\sum_{i=1}^{N_n}\E C_i\leq CN_n \Big|a-\frac{N_n}{N_n+M_n}\Big|\Big(\E\big[X_t^{i}-\widehat{X_t^{i}}\big]^2\Big)^{\frac{1}{2}}= CN_n \Big|a-\frac{N_n}{N_n+M_n}\Big|\sqrt{\omega(t)}.
\end{equation}
Now we estimate $D_i$. This is a cross term that involves both species and we will need to use assumptions on the Lipschitz property of $F_{12}$. We   first add and subtract appropriate terms similarly as in $B_i$.
\begin{align}
\label{eq: Di}
D_i(t)&=\bigg[-\frac{1}{N_n+M_n}\sum_{k=1}^{M_n}\Big(\nabla F_{12}(X_t^{i}-Y_t^{k})-(\nabla F_{12}\ast\nu_t)(\widehat{X_t^{i}})\Big)\bigg]\cdot (X_{t}^{i}-\widehat{X_t^{i}})\notag
\\&=\bigg[-\frac{1}{N_n+M_n}\sum_{k=1}^{M_n}\Big(\nabla F_{12}(X_t^{i}-Y_t^{k})-\nabla F_{12}(\widehat{X_t^{i}}-Y_t^{k})\Big)\bigg]\cdot (X_{t}^{i}-\widehat{X_t^{i}})\notag
\\&\qquad-\bigg[\frac{1}{N_n+M_n}\sum_{k=1}^{M_n}\Big(\nabla F_{12}(\widehat{X_t^{i}}-Y_t^{k})-(\nabla F_{12}\ast\nu_t)(\widehat{X_t^{i}})\Big)\bigg]\cdot (X_{t}^{i}-\widehat{X_t^{i}})\notag
\\&=: D_i^1(t)+D_i^2(t).
\end{align}
Using Lipschitzian property of $\nabla F_{12}$ we have
\[
|\nabla F_{12}(X_t^{i}-Y_t^{k})-\nabla F_{12}(\widehat{X_t^{i}}-Y_t^{k})|\leq K|X_t^{i}-\widehat{X_t^{i}}|,
\] 
which implies that
\begin{equation}
\label{eq: D1i}
\E\left[D^1_i(t)\right]\leq \frac{CM_n}{N_n+M_n}\E \big(X_t^{i}-\widehat{X_t^{i}}\big)^2.
\end{equation}
Taking the expectation of $D^2_i(t)$, noting that $\nu_t=\mathrm{Law}(\widehat{Y_t^{k}})$, we obtain
\begin{align*}
\E\left[D_i^2(t)\right]=\E\bigg[-\frac{1}{N_n+M_n}\sum_{k=1}^{M_n}\Big(\nabla F_{12}(\widehat{X_t^{i}}-Y_t^{k})-(\nabla F_{12}(\widehat{X_t^{i}}-\widehat{Y_t^{k}})\Big)\bigg]\cdot (X_{t}^{i}-\widehat{X_t^{i}}).
\end{align*}
Then similarly as in $D^1_i(t)$, we have
\begin{align*}
\E\left[D^2_i(t)\right] \leq \frac{C}{N_n+M_n}\sum_{k=1}^{M_n}\E\Big(|Y_t^{k}-\widehat{Y_t^{k}}||X_t^{i}-\widehat{X_t^{i}}|\Big),
\end{align*}
which implies that
\begin{align}
\label{eq: D2i}
\sum_{i=1}^{N_n}\E\left[D^2_i(t)\right]&\leq \frac{C}{N_n+M_n}\E\Big(\sum_{i=1}^{N_n}|X_t^{i}-\widehat{X_t^{i}}|\sum_{k=1}^{M_n}|Y_t^{k}-\widehat{Y_t^{k}}|\Big)
\notag
\\&\leq \frac{C\sqrt{N_n}\sqrt{M_n}}{N_n+M_n}\Big(\sum_{i=1}^{N_n}\E\big(X_t^{i}-\widehat{X_t^{i}}\big)^2\Big)^\frac{1}{2}\Big(\sum_{k=1}^{M_n}\E\big(Y_t^{k}-\widehat{Y_t^{k}}\big)^2\Big)^\frac{1}{2}\notag
\\&=\frac{CM_n N_n}{N_n+M_n}\Big(\E\big(X_t^{i}-\widehat{X_t^{i}}\big)^2\Big)^\frac{1}{2}\Big(\E\big(Y_t^{k}-\widehat{Y_t^{k}}\big)^2\Big)^\frac{1}{2}\notag
\\&=\frac{CM_n N_n}{N_n+M_n}\sqrt{\omega(t)\widehat{\omega}(t)}.
\end{align}
Substituting \eqref{eq: D1i} and \eqref{eq: D2i} into \eqref{eq: Di} we obtain
\begin{equation}
\label{eq: Di2}
\sum_{i=1}^{N_n}\E\left[D_i(t)\right] \leq \frac{C M_n N_n}{N_n+M_n}\Big(\omega(t)+\sqrt{\omega(t)\widehat{\omega}(t)}\Big).
\end{equation}
Finally we estimate $E_i$ analogously as in $C_i$ and get
\begin{align*}
\E\left[E_i(t)\right]\leq C \Big|(1-a)-\frac{M_n}{N_n+M_n}\Big|\Big(\E\big(X_t^{i}-\widehat{X_t^{i}}\big)^2\Big)^{\frac{1}{2}}= C \Big|(1-a)-\frac{M_n}{N_n+M_n}\Big|\sqrt{\omega(t)}.
\end{align*}
Taking the sum over $i$ from $1$ to $N_n$ yields
\begin{equation}
\label{eq: Ei2}
\sum_{i=1}^{N_n}\E\left[E_i(t)\right] \leq CN_n \Big|(1-a)-\frac{M_n}{N_n+M_n}\Big|\sqrt{\omega(t)}.
\end{equation}
Substituting \eqref{eq: Ai}, \eqref{eq: Bi-2},\eqref{eq: Di2} and \eqref{eq: Ei2} into \eqref{eq: est1} we obtain
\begin{align*}
N_n\omega(t)&\leq 2 \int_0^t\bigg(N_n\theta_1\omega(s)+\frac{C N_n^{3/2}}{N_n+M_n}\sqrt{\omega(s)}+CN_n\Big|a-\frac{N_n}{N_n+M_n}\Big|\sqrt{\omega(t)}+\frac{CM_n N_n}{N_n+M_n}\omega(s)
\\&\hspace*{2cm}+\frac{CM_nN_n}{N_n+M_n}\sqrt{\omega(s)\widehat{\omega}(s)}+CN_n\Big|(1-a)-\frac{M_n}{N_n+M_n}\Big|\sqrt{\omega(t)}\bigg)\,ds.
\end{align*}
By dividing both sides by $N_n$ we get
\begin{align}
\label{eq: omega}
\omega(t)&\leq C \int_0^t\bigg(\omega(s)+\Big(\frac{N_n^{1/2}}{N_n+M_n}+\Big|a-\frac{N_n}{N_n+M_n}\Big|+\Big|(1-a)-\frac{M_n}{N_n+M_n}\Big|\Big)\sqrt{\omega(s)}
\\&\hspace*{2cm} +\sqrt{\omega(s)\widehat{\omega}(s)}\bigg)\,ds\notag
\\&\leq C\int_0^t\Big(\omega(s)+\widehat{\omega}(s)+f(n)\sqrt{\omega(s)+\widehat{\omega}(s)}\Big)\,ds,
\end{align}
where $0\leq f(n)\leq C \frac{N_n^{1/2}}{N_n+M_n}+\Big|a-\frac{N_n}{N_n+M_n}\Big|+\Big|(1-a)-\frac{M_n}{N_n+M_n}\Big|$ (hence $f(n)\to 0$ as $n\to 0$).  Note that we have used $\widehat{\omega}(s)\geq 0$ and the elementary $\sqrt{xy}\leq \frac{1}{2}(x+y)$ to obtain the last estimate.

Analogously we obtain 
\begin{equation}
\label{eq: homega}
\widehat{\omega}(t)\leq C\int_0^t\Big(\omega(s)+\widehat{\omega}(s)+\widehat{f}(n)\sqrt{\omega(s)+\widehat{\omega}(s)}\Big)\,ds
\end{equation}
for some function $0\leq \widehat{f}(n)$ that tends to $0$ as $n$ goes to infinity.

Taking the sum of \eqref{eq: omega} and \eqref{eq: homega} yields
\begin{equation}
\omega(t)+\widehat{\omega}(t)\leq C \int_0^t\Big(\omega(s)+\widehat{\omega}(s)+(f(n)+\widehat{f}(n))\sqrt{\omega(s)+\widehat{\omega}(s)}\Big)\,ds.
\end{equation}
Applying Lemma~\ref{lem: Gronwall}, we obtain
\begin{equation}
\omega(t)+\widehat{\omega}(t)\leq C\bigg(\big(f(n)+\widehat{f}(n)\big)\Big(e^\frac{Ct}{2}-1\Big)\bigg)^2.
\end{equation}
Since $f(n)+\widehat{f}(n)\to 0$ as $n\to 0$, the last estimate implies that
\[
\lim\limits_{n\to\infty}\sup_{t\in [0,T]}\omega(t)=0\quad\text{and}\quad \lim\limits_{n\to\infty}\sup_{t\in [0,T]}\widehat{\omega}(t)=0.
\]
This completes the proof of Proposition~\ref{prop: 2nd}.
\end{proof}
\begin{remark} If $N_n$ and $M_n$ tend to $+\infty$ simultaneously but $\frac{N_n}{M_n}$ is a constant, then from the computations in the proof  of Proposition \ref{prop: 2nd}, we obtain explicit  estimates
\begin{equation}
\sup_{t\in[0,T]}\E\big(X_t^{i}-\widehat{X_t^{i}}\big)^2\leq \frac{C}{N_n}~~\text{and}~~\sup_{t\in[0,T]}\E\big(Y_t^{i}-\widehat{Y_t^{i}}\big)^2\leq \frac{\hat C}{M_n},
\end{equation}
for some positive constants $C$ an $\hat C$.
\end{remark}

The following proposition strengthens Proposition~\ref{prop: 2nd}. 
\begin{proposition}
\label{prop: 4th}
We have
\begin{equation}
\lim\limits_{n\to\infty}\sup\limits_{t\in [0,T]}\E\bigg[\Big(X_t^{i}-\widehat{X_t^{i}}\Big)^4\bigg]=0\quad\text{and}\quad\lim\limits_{n\to\infty}\sup\limits_{t\in [0,T]}\E\bigg[\Big(Y_t^{i}-\widehat{Y_t^{i}}\Big)^4\bigg]=0.\end{equation}
\end{proposition}
\begin{proof}
Let us define
\begin{equation}
\zeta(t):=\E\Big[(X_t^{i}-\widehat{X_t^{i}})^4\Big]\quad\text{and}\quad \widehat{\zeta}(t):=\E\Big[(Y_t^{i}-\widehat{Y_t^{i}})^4\Big].
\end{equation}
The strategy of the proof will be similar to that of Proposition~\ref{prop: 2nd} that consists of three steps: (1) using It\^{o}'s lemma to obtain an expression for $\zeta(t)$, (2) estimating each term that appears in the expression to derive a Gr\"onwall type inequality for $\zeta(t)$ and (3) applying Lemma \ref{lem: Gronwall} to deduce the assertion. 

We now carry out this procedure and will refer to the proof of Proposition~\ref{prop: 2nd} when similar arguments apply. We first use It\^{o} formula to obtain
\begin{align}
\zeta(t)&=\E\Big[(X_t^{i}-\widehat{X_t^{i}})^4\Big]
\\&= -4\E\int_0^t (\nabla V_1(X_s^{i})-\nabla V_1(\widehat{X_s^{i}}))\cdot (X_s^{i}-\widehat{X_s^{i}})^3 \,ds\notag
\\&\qquad-4\E\int_0^t \Big[\frac{1}{N_n+M_n}\sum_{j=1}^{N_n}\nabla F_{11}(X_s^{i}-X_s^{j})-a(\nabla F_{11}\ast\mu_s)(\widehat{X_s^{i}})\Big]\cdot (X_s^{i}-\widehat{X_s^{i}})^3 \,ds\notag
\\&\qquad-4\E\int_0^t \Big[\frac{1}{N_n+M_n}\sum_{k=1}^{M_n}\nabla F_{12}(X_s^{i}-Y_s^{k})-(1-a)(\nabla F_{12}\ast\mu_s)(\widehat{X_s^{i}})\Big]\cdot (X_s^{i}-\widehat{X_s^{i}})^3 \,ds\notag
\\&=:4\int_0^t(F_i(s)+G_i(s)+H_i(s)+I_i(s))\,ds,
\label{eq: est4-1}
\end{align} 
where
\begin{align*}
F_i(t)&:=-\Big(\nabla V_1(X_t^{i})-\nabla V_1(\widehat{X_t^{i}})\Big)\cdot(X_t^{i}-\widehat{X_t^{i}})^3,\\
G_i(t)&:=\bigg[-\frac{1}{N_n+M_n}\sum_{j=1}^{N_n}\Big(\nabla F_{11}(X_t^{i}-X_t^{j})-(\nabla F_{11}\ast\mu_t)(\widehat{X_t^{i}})\Big)\bigg]\cdot (X_{t}^{i}-\widehat{X_t^{i}})^3,
\\H_i(t)&:=\Big(a-\frac{N_n}{N_n+M_n}\Big)(\nabla F_{11}\ast \mu_t)(\widehat{X_t^{i}})\cdot(X_t^{i}-\widehat{X_t^{i}})^3,
\\I_i(t)&:=\bigg[-\frac{1}{N_n+M_n}\sum_{k=1}^{M_n}\Big(\nabla F_{12}(X_t^{i}-Y_t^{k})-(\nabla F_{12}\ast\nu_t)(\widehat{X_t^{i}})\Big)\bigg]\cdot (X_{t}^{i}-\widehat{X_t^{i}})^3,
\\J_i(t)&:=\Big((1-a)-\frac{M_n}{N_n+M_n}\Big)(\nabla F_{12}\ast\nu_t)(\widehat{X_t^{i}})\cdot(X_t^{i}-\widehat{X_t^{i}})^3.
\end{align*}
Next we estimate each term $F_i, G_i, H_i$ and $I_i$. According to Assumption~\ref{asp: assumption}, we have
\begin{equation}
\label{eq: F}
\E\left[F_i(s)\right]\leq \theta_1 \E\left[(X_s^{i}-\widehat{X_s^{i}})^4\right].
\end{equation}
We write $G_i(t)=-\frac{1}{N_n+M_n}(G^1_i(t)+G^2_i(t))$, where
\begin{align*}
G^1_i(t)&:=\sum_{j=1}^{N_n}\Big(\nabla F_{11}(X_t^{i}-X_t^{j})-\nabla F_{11}(\widehat{X_t^{i}}-\widehat{X_t^{j}}\Big)\cdot (X_{t}^{i}-\widehat{X_t^{i}})^3,\\
G^2_i(t)&:=\sum_{j=1}^{N_n}\Big(\nabla F_{11}(\widehat{X_t^{i}}-\widehat{X_t^{j}})-(\nabla F_{11}\ast\mu_t)(\widehat{X_t^{i}}\Big)\cdot (X_{t}^{i}-\widehat{X_t^{i}})^3.
\end{align*}
Similarly as in the proof of Proposition~\ref{prop: 2nd}, we have
\[
\sum_{i=1}^{N_n}G^1_i(t)\geq 0.
\]
Using H\"{o}lder's inequality
\begin{equation}
\label{eq: Holder inequality}
\Big|\int f g\, d\gamma \Big|\leq \Big(\int |f|^{4/3}\,d\gamma\Big)^{3/4}\Big(\int |g|^{4}\,d\gamma\Big)^{1/4},
\end{equation}
we get
\begin{align*}
\E\left[G_i^2(t)\right]&=\E\Big[\sum_{j=1}^{N_n}\Big(\nabla F_{11}(\widehat{X_t^{i}}-\widehat{X_t^{j}})-(\nabla F_{11}\ast\mu_t)(\widehat{X_t^{i}}\Big)\cdot (X_{t}^{i}-\widehat{X_t^{i}})^3\Big]
\\&\qquad\leq \Big(\E\big[(X_{t}^{i}-\widehat{X_t^{i}})^4\big]\Big)^{3/4}(\widehat{\kappa}_i(t))^{1/4},
\end{align*}
where
\[
\widehat{\kappa}_i(t):=\E\bigg[\Big(\sum_{j=1}^{N_n}\nabla F_{11}(\widehat{X_t^{i}}-\widehat{X_t^{j}})-(\nabla F_{11}\ast\mu_t)(\widehat{X_t^{i}})\Big)^{4}\bigg]=:\E\bigg[\Big(\sum_{j=1}^{N_n}a_j(t)\Big)^{4}\bigg]
\]
where $a_j(t):=\nabla F_{11}(\widehat{X_t^{i}}-\widehat{X_t^{j}})-(\nabla F_{11}\ast\mu_t)(\widehat{X_t^{i}})$. Using the multinomial theorem
\begin{align*}
\Big(\sum_{j=1}^n a_j\Big)^4=&\sum_{j=1}^n a_j^4~~+\sum_{1\leq j\neq k\leq n}4 a_j^3 a_k+\sum_{1\leq j<k\leq n}6 a_j^2a_k^2\\
&+\sum_{1\leq j\neq k\neq \ell\leq n}12 a_j^2 a_j a_{\ell}~~+\sum_{1\leq j\neq k\neq \ell\neq m\leq n}24 a_j a_k a_{\ell} a_m,
\end{align*}
we decompose $\widehat{\kappa}_i(t)$ as follows
\[
\widehat{\kappa}_i(t)=\widehat{\kappa}^{(1)}_i(t)+\widehat{\kappa}^{(2)}_i(t)+\widehat{\kappa}^{(3)}_i(t)+\widehat{\kappa}^{(4)}_i(t)+\widehat{\kappa}^{(5)}_i(t),
\]
where
\begin{align*}
&\widehat{\kappa}^{(1)}_i(t)=\E\Big(\sum_{j=1}^n a_j(t)^4\Big),\quad  \widehat{\kappa}^{(2)}_i(t)=\E\Big(\sum_{1\leq j\neq k\leq n}4 a_j(t)^3 a_k(t)\Big), \quad \widehat{\kappa}^{(3)}_i(t)=\E\Big(\sum_{1\leq j<k\leq n}6 a_j(t)^2a_k(t)^2\Big),
\\&\widehat{\kappa}^{(4)}_i(t)=\E\Big(\sum_{1\leq j\neq k\neq \ell\leq n}12 a_j(t)^2 a_j(t) a_{\ell}(t)\Big), \quad \widehat{\kappa}^{(5)}_i(t)=\E\Big(\sum_{1\leq j\neq k\neq \ell\neq m\leq n}24 a_j(t)a_k(t)a_{\ell}(t)a_m(t)\Big).
\end{align*}
As in the proof of Proposition~\ref{prop: 2nd}
\[
\widehat{\kappa}^{(2)}_i(t)=\widehat{\kappa}^{(4)}_i(t)=\widehat{\kappa}^{(5)}_i(t)=0\quad\text{and}\quad \widehat{\kappa}^{(1)}_i(t)\leq CN_n^2,\quad \widehat{\kappa}^{(3)}_i(t)\leq CN_n^2.
\]
Therefore, we obtain $\widehat{\kappa_i}(t)\leq CN_n^2$, thus
\begin{equation}
\label{eq: G}
\E\left[G^2_i(t)\right]\leq C \sqrt{N_n}	\Big(\E\big[(X_{t}^{i}-\widehat{X_t^{i}})^4\big]\Big)^{3/4}\quad\text{and}\quad \E\left[G_i(t)\right]\leq \frac{C\sqrt{N_n}}{N_n+M_n}\Big(\E\big[(X_{t}^{i}-\widehat{X_t^{i}})^4\big]\Big)^{3/4}.
\end{equation}
Next we estimate $H_i$. Using H\"{o}lder's inequality \eqref{eq: Holder inequality} again we have
\begin{align*}
\E\left[H_i(t)\right]\leq \Big|a-\frac{N_n}{N_n+M_n}\Big|\Big(\E\big[((X_{t}^{i}-\widehat{X_t^{i}})^4\big]\Big)^{3/4}\Big(\E\big[(\nabla F_{11}\ast \mu_t)(\widehat{X_t^{i}})^4\big] \Big)^{1/4}. 
\end{align*}
As in the proof of \eqref{eq: bound convol term}, it holds that
\[
\E\big[(\nabla F_{11}\ast \mu_t)(\widehat{X_t^{i}})^4\big]\leq C,
\]
which implies that
\begin{equation}
\label{eq: H}
\E(H_i)\leq C \Big|a-\frac{N_n}{N_n+M_n}\Big|\Big(\E\big[((X_{t}^{i}-\widehat{X_t^{i}})^4\big]\Big)^{3/4}.
\end{equation}
We proceed with the term $I_i$
\begin{align}
\label{eq: Ii-1}
I_i(t)&=\bigg[-\frac{1}{N_n+M_n}\sum_{k=1}^{M_n}\Big(\nabla F_{12}(X_t^{i}-Y_t^{k})-(\nabla F_{12}\ast\nu_t)(\widehat{X_t^{i}})\Big)\bigg]\cdot (X_{t}^{i}-\widehat{X_t^{i}})^3\notag
\\&=\bigg[-\frac{1}{N_n+M_n}\sum_{k=1}^{M_n}\Big(\nabla F_{12}(X_t^{i}-Y_t^{k})-\nabla F_{12}(\widehat{X_t^{i}}-Y_t^{k})\Big)\bigg]\cdot (X_{t}^{i}-\widehat{X_t^{i}})^3\notag
\\&\qquad-\bigg[\frac{1}{N_n+M_n}\sum_{k=1}^{M_n}\Big(\nabla F_{12}(\widehat{X_t^{i}}-Y_t^{k})-(\nabla F_{12}\ast\nu_t)(\widehat{X_t^{i}})\Big)\bigg]\cdot (X_{t}^{i}-\widehat{X_t^{i}})^3\notag
\\&=: I_i^{(1)}(t)+I_i^{(2)}(t).
\end{align}
Using the Lipschitz property of $\nabla F_{12}$ we get
\begin{align*}
|I_i^{(1)}(t)|\leq \frac{K}{N_n+M_n}\sum_{k=1}^{M_n}(X_t^{i}-\widehat{X_t^{i}})^4,
\end{align*}
which implies that
\begin{equation*}
\E\left[I_i^{(1)}(t)\right]\leq \frac{KM_n}{N_n+M_n}\E\Big[(X_t^{i}-\widehat{X_t^{i}})^4\Big].
\end{equation*}
Summing this estimate over $i$ yields
\begin{equation}
\label{eq: Ii1}
\sum_{i=1}^{N_n}\E\left[I_i^{(1)}(t)\right]\leq \frac{KM_n N_n}{N_n+M_n}\E\Big[(X_t^{i}-\widehat{X_t^{i}})^4\Big].
\end{equation}
For the term $I_i^{(2)}(t)$, we have
\begin{align*}
\E\left[I_i^{(2)}(t)\right]&=-\frac{1}{N_n+M_n}\E\Big[\sum_{k=1}^{M_n}\Big(\nabla F_{12}(\widehat{X_t^{i}}-Y_t^{k})-(\nabla F_{12}\ast\nu_t)(\widehat{X_t^{i}})\cdot (X_{t}^{i}-\widehat{X_t^{i}})^3\Big]
\\&=-\frac{1}{N_n+M_n}\E\Big[\sum_{k=1}^{M_n}\Big(\nabla F_{12}(\widehat{X_t^{i}}-Y_t^{k})-\nabla F_{12}(\widehat{X_t^{i}}-\widehat{Y_t^{k}})\cdot (X_{t}^{i}-\widehat{X_t^{i}})^3\Big]
\\&\leq\frac{K}{M_n+N_n}\sum_{k=1}^{M_n}\E\Big(|Y_t^{k}-\widehat{Y_t^{k}}||X_{t}^{1,i}-\widehat{X_t^{i}}|^3\Big)
\\&\leq \frac{K}{M_n+N_n}\sum_{k=1}^{M_n}\Big(\E(Y_t^{k}-\widehat{Y_t^{k}})^4\Big)^{1/4}\Big(\E(X_t^{i}-\widehat{X_t^{i}})^4\Big)^{3/4}.
\end{align*}
Hence
\begin{align}
\label{eq: Ii2}
\sum_{i=1}^{N_n}\E\left[I_i^{(2)}(t)\right]&\leq \frac{K}{M_n+N_n}\Big(\sum_{k=1}^{M_n}\Big(\E(Y_t^{k}-\widehat{Y_t^{k}})^4\Big)^{1/4}\Big)\Big(\sum_{i=1}^{N_n}\Big(\E(X_t^{i}-\widehat{X_t^{i}})^4\Big)^{3/4}\Big)\notag
\\&=\frac{KM_nN_n}{M_n+N_n}\Big(\E(Y_t^{k}-\widehat{Y_t^{k}})^4\Big)^{1/4}\Big)\Big(\E(X_t^{i}-\widehat{X_t^{i}})^4\Big)^{3/4}\Big).
\end{align}
Substituting \eqref{eq: Ii1} and \eqref{eq: Ii2} into \eqref{eq: Ii-1}, we obtain
\begin{align}
\label{eq: I}
\sum_{i=1}^{N_n}\E\left[I_i(t)\right]\leq& \frac{KM_n N_n}{N_n+M_n}\E\Big[(X_t^{i}-\widehat{X_t^{i}})^4\Big]\\
&+\frac{KM_nN_n}{M_n+N_n}\Big(\E(Y_t^{k}-\widehat{Y_t^{k}})^4\Big)^{1/4}\Big)\Big(\E(X_t^{i}-\widehat{X_t^{i}})^4\Big)^{3/4}\Big).\nonumber
\end{align}
Similarly as $H_i$ we get
\begin{equation}
\label{eq: J}
\E\left[J_i(t)\right]\leq C\Big|(1-a)-\frac{M_n}{N_n+M_n}\Big|
\Big(\E\big[((X_{t}^{i}-\widehat{X_t^{i}})^4\big]\Big)^{3/4}.
\end{equation}
Taking the sum over $i$ in \eqref{eq: est4-1} and from estimates \eqref{eq: F}, \eqref{eq: G}, \eqref{eq: H} and \eqref{eq: I} we get
\begin{align}
N_n\zeta(t)\leq 4\int_0^t\Big[&N_n\theta_1 \zeta(s)+\frac{CN_n^{3/2}}{N_n+M_n}\zeta^{3/4}(s)+N_n\Big|a-\frac{N_n}{N_n+M_n}\Big|\zeta^{3/4}(s)\notag
\\&+C N_n\Big|(1-a)-\frac{M_n}{N_n+M_n}\Big|\zeta^{3/4}(s)+\frac{KM_n N_n}{N_n+M_n}\zeta(s)+\frac{KM_nN_n}{M_n+N_n}\zeta^{3/4}(s)\widehat{\zeta}^{1/4}(s)
\Big]. \notag
\end{align}
By dividing both sides of the above estimate by $N_n$, noting that $\frac{M_n}{M_n+N_n}\leq 1$,
we write the result in a compact form 
\begin{equation}
\zeta(t)\leq C\int_0^t \Big(\zeta(s)+g(n)\zeta^{3/4}(s)+\zeta^{3/4}\widehat{\zeta}^{1/4}(s)\Big)\,ds
\end{equation}
where
\[
g(n)=\frac{\sqrt{N_n}}{N_n+M_n}+ \Big|a-\frac{N_n}{N_n+M_n}\Big|+\Big|(1-a)-\frac{M_n}{N_n+M_n}\Big|\overset{n\to\infty}{\longrightarrow} 0
\]
Using the inequality of arithmetic and geometric means,
\[
3a+b=a+a+a+b\geq 4 a^{3/4}b^{1/4}\quad \text{for all}~~a,b\geq 0,
\]
and the non-negativity of $\zeta$ and $\widehat{\zeta}$, we obtain the following estimate
\begin{equation}
\label{eq: zeta}
\zeta(t)\leq C \int _0^t \Big[\zeta(s)+\widehat{\zeta}(s)+g(n)(\zeta(s)+\widehat{\zeta}(s))^{3/4}\Big]\,ds.
\end{equation}
Analogously we obtain a similar estimate for $\widehat{\zeta}$
\begin{equation}
\label{eq: hat zeta}
\widehat{\zeta}(t)\leq C \int _0^t \Big[\zeta(s)+\widehat{\zeta}(s)+\widehat{g}(n)(\zeta(s)+\widehat{\zeta}(s))^{3/4}\Big]\,ds,
\end{equation}
where
\[
\widehat{g}(n)=\frac{\sqrt{M_n}}{N_n+M_n}+ \Big|a-\frac{N_n}{N_n+M_n}\Big|+\Big|(1-a)-\frac{M_n}{N_n+M_n}\Big| \overset{n\to\infty}{\longrightarrow}0.
\]
Adding \eqref{eq: zeta} and \eqref{eq: hat zeta} gives
\begin{equation}
(\zeta+\widehat{\zeta})(t)\leq C \int _0^t \Big[\zeta(s)+\widehat{\zeta}(s)+(g(n)+\widehat{g}(n))(\zeta(s)+\widehat{\zeta}(s))^{3/4}\Big]\,ds.
\end{equation}
Applying Lemma \ref{lem: Gronwall} for $\alpha=\frac{3}{4}$ we get
\begin{equation}
(\zeta+\widehat{\zeta})(t)\leq C\Big((g(n)+\widehat{g}(n))\big(e^{Ct}-1\big)\Big)^4,
\end{equation}
from which we deduce the statement of the proposition.
\end{proof}
\begin{remark} If $N_n$ and $M_n$ tend to $+\infty$ simultaneously but $\frac{N_n}{M_n}$ is a constant, then from the computations in the proof  of Proposition \ref{prop: 2nd}, we obtain explicit  estimates
\begin{equation}
\sup_{t\in[0,T]}\E\left[\big(X_t^{i}-\widehat{X_t^{i}}\big)^4\right]\leq \frac{C}{N^2_n}~~\text{and}~~\sup_{t\in[0,T]}\E\left[\big(Y_t^{i}-\widehat{Y_t^{i}}\big)^4\right]\leq \frac{\hat C}{M_n^2},
\end{equation}
for some positive constants $C$ an $\hat C$.
\end{remark}

We are now ready to prove Theorem \ref{theo: PoC}.
\begin{proof}[\textbf{Proof of Theorem \ref{theo: PoC}}]
According to \eqref{eq: est1} we have
\begin{align}
&\big(X_t^{i}-\widehat{X_t^{i}}\big)^2\notag\\
&=-2\int_0^t (\nabla V_1(X_s^{i})-\nabla V_1(\widehat{X_s^{i}}))\cdot (X_s^{i}-\widehat{X_s^{i}}) \,ds\notag
\\&\quad-2\int_0^t \Big[\frac{1}{N_n+M_n}\sum_{j=1}^{N_n}\nabla F_{11}(X_s^{i}-X_s^{j})-a(\nabla F_{11}\ast\mu_s)(\widehat{X_s^{i}})\Big]\cdot (X_s^{i}-\widehat{X_s^{i}}) \,ds\notag
\\&\quad-2\int_0^t \Big[\frac{1}{N_n+M_n}\sum_{k=1}^{M_n}\nabla F_{12}(X_s^{i}-Y_s^{k})-(1-a)(\nabla F_{12}\ast\mu_s)(\widehat{X_s^{i}})\Big]\cdot (X_s^{i}-\widehat{X_s^{i}}) \,ds.
\label{eq: estimate1}
\end{align}
We define
\begin{align}
L&:=2\int_0^T\bigg|\Big[\frac{1}{N_n+M_n}\sum_{j=1}^{N_n}\nabla F_{11}(X_s^{i}-X_s^{j})-a(\nabla F_{11}\ast\mu_s)(\widehat{X_s^{i}})\Big]\cdot (X_s^{i}-\widehat{X_s^{i}})\bigg| \,ds\notag
\\&\quad+2\int_0^T \bigg|\Big[\frac{1}{N_n+M_n}\sum_{k=1}^{M_n}\nabla F_{12}(X_s^{i}-Y_s^{k})-(1-a)(\nabla F_{12}\ast\mu_s)(\widehat{X_s^{i}})\Big]\cdot (X_s^{i}-\widehat{X_s^{i}})\bigg| \,ds
\\&=:L_1+L_2.
\end{align}
From \eqref{eq: estimate1} and \eqref{eq: V1} we can estimate
\begin{equation}
\label{eq: estimate2}
\big(X_t^{i}-\widehat{X_t^{i}}\big)^2\leq 2\theta_1\int_0^t \Big|X_s^{i}-\widehat{X_s^{i}}\Big|^2\, ds+L.
\end{equation}
Setting $\varphi(t):=\int_0^t \Big|X_s^{i}-\widehat{X_s^{i}}\Big|^2\, ds$, we get
\[
\varphi'(t)\leq 2\theta_1 \varphi(t)+L.
\]
Using Gr\"onwall lemma and the fact that $ \varphi(0) = 0$, we deduce that, for any $t\in[0,T]$,
\[
\varphi(t)\leq \frac{L}{2\theta_1}\Big(e^{2\theta_1 t}-1\Big).
\]
Substituting this back into \eqref{eq: estimate2}, we obtain
\begin{equation*}
\big(X_t^{i}-\widehat{X_t^{i}}\big)^2=\varphi'(t)\leq M e^{2\theta_1 t}\leq e^{2\theta_1 T}L,
\end{equation*}
from which we deduce that
\begin{equation}
\label{eq: estimate 3}
\E\Big[\sup_{t\in [0,T]}\big(X_t^{i}-\widehat{X_t^{i}}\big)^2\Big]\leq e^{2\theta_1 T}\E\left[L\right].
\end{equation}
Next we find an upper bound for $L$. For the first term, $L_1$, proceeding similarly as the terms $B_i$ and $C_i$ in the proof of Proposition~\ref{prop: 2nd}, we have
\begin{align}
\label{eq: L1}
\E\left[L_1\right]&\leq \frac{2}{M_n+N_n}\sum_{j=1}^{N_n}\int_0^T \big(|\rho^1_{ij}(s)|+|\rho^2_{ij}(s)|\big)\,ds+ 2C\Big|a-\frac{N_n}{M_n+N_n}\Big|\int_0^T \Big(E\big[X_s^{i}-\widehat{X}_s^{i}\big]^2\Big)^\frac{1}{2}\, ds\notag
\\&\leq \frac{2}{M_n+N_n}\sum_{j=1}^{N_n}\int_0^T \big(|\rho^1_{ij}(s)|+|\rho^2_{ij}(s)|\big)\,ds+ 2C T\Big|a-\frac{N_n}{M_n+N_n}\Big| \Big(\sup_{s\in [0,T]} E\big[X_s^{i}-\widehat{X}_s^{i}\big]^2\Big)^\frac{1}{2}\notag
\\&=: K_1+K_2,
\end{align}
where $\rho^1_{ij}(s)$ and $\rho^2_{ij}(s)$ are defined in \eqref{eq: rhoij}.

By Proposition \ref{prop: 2nd}, $K_2$ tends to $0$ as $n$ goes to $+\infty$. If $\frac{M_n}{N_n}$ is constant, then $K_2=0$. We now estimate two terms in $K_1$. According to \eqref{eq: rho2ij}, we have
\begin{equation*}
E\Big(\sum_{j=1}^{N_n}|\rho^2_{ij}(s)|\Big)\leq C\sqrt{N_n}\Big(\sup_{s\in[0,T]}E\big[X_s^{i}-\widehat{X}_s^{i}\big]^2\Big)^\frac{1}{2}. 
\end{equation*}
Thus,
\begin{equation*}
\frac{2}{M_n+N_n}\int_0^TE\Big(\sum_{j=1}^{N_n}|\rho^2_{ij}(s)|\Big)\,ds\leq \frac{2 CT\sqrt{N_n}}{M_n+N_n}\Big(\sup_{s\in[0,T]}E\big[X_s^{i}-\widehat{X}_s^{i}\big]^2\Big)^\frac{1}{2},
\end{equation*}
which converges to $0$ as $n\rightarrow +\infty$. If $\frac{M_n}{N_n}$ is constant, then 
\[
\sup_{s\in[0,T]}E\left[\big(X_s^{i}-\widehat{X}_s^{i}\big)^2\right]\leq \frac{C}{N_n},
\]
which implies that
\begin{equation}
\label{eq: Erho2}
\frac{2}{M_n+N_n}\int_0^TE\Big(\sum_{j=1}^{N_n}|\rho^2_{ij}(s)|\Big)\,ds\leq \frac{2 CT}{N_n}.
\end{equation}
For the $\rho^1_{ij}$ term, using \eqref{eq: rhoij}, Cauchy-Schwarz inequality and Proposition \ref{arthur}, we get
\begin{align*}
\E\big[|\rho^1_{ij}(s)\big]&\leq \bigg\{\E\big[(X_s^{i}-\widehat{X}_s^{i})^2\big] \E\Big[\big(\nabla F_{11}(X_t^{i}-X_t^{j})-\nabla F_{11}(\widehat{X_t^{i}}-\widehat{X_t^{j}})\big)^2\Big]\bigg\}^\frac{1}{2}
\\&\leq \bigg\{\E\big[(X_s^{i}-\widehat{X}_s^{i})^2\bigg\}^\frac{1}{2}
\\&\qquad\qquad\times \bigg\{\E\Big[\big(X^{i}_s-\widehat{X}^{i}_s + \widehat{X}^{j}_s-X_s^{j}\big)^2\big(c+|X^{i}_s-X_s^{j}|^{2q}+|\widehat{X}^{i}_s-\widehat{X}^{j}_s|^{2q}\big)^2\Big] \bigg\}^\frac{1}{2}
\\&\leq \bigg\{\sup_{s\in[0,T]} \E\big[(X_s^{i}-\widehat{X}_s^{i})^2\bigg\}^\frac{1}{2}
\\&\qquad\times \bigg\{ \E\Big[\big(X^{i}_s-\widehat{X}^{i}_s + \widehat{X}^{j}_s-X_s^{j}\big)^4\Big] \E\Big[\big(c+|X^{i}_s-X_s^{j}|^{2q}+|\widehat{X}^{i}_s-\widehat{X}^{j}_s|^{2q}\big)^4\Big]\bigg\}^{\frac{1}{4}}
\\&\leq C \bigg\{\sup_{s\in[0,T]} \E\big[(X_s^{i}-\widehat{X}_s^{i})^2\bigg\}^\frac{1}{2}\times \bigg\{\sup_{s\in[0,T]} \E\big[(X_s^{i}-\widehat{X}_s^{i})^4\bigg\}^\frac{1}{4},
\end{align*}
which tends to $0$ as $n$ tends to $+\infty$ according to Propositions \ref{prop: 2nd} and \ref{prop: 4th}. If $\frac{N_n}{M_n}$ is constant then,
\begin{equation*}
\E\big[|\rho^1_{ij}(s)|\big]\leq C \bigg\{\sup_{s\in[0,T]} \E\big[(X_s^{i}-\widehat{X}_s^{i})^2\bigg\}^\frac{1}{2}\times \bigg\{\sup_{s\in[0,T]} \E\big[(X_s^{i}-\widehat{X}_s^{i})^4\bigg\}^\frac{1}{4}\leq \frac{C}{N_n}.
\end{equation*}
Thus
\begin{equation}
\label{eq: Erho1}
\frac{2}{M_n+N_n}\sum_{j=1}^{N_n}\int_0^T \E\big[|\rho^1_{ij}(s)|\big]\,ds \leq \frac{C}{N_n}.
\end{equation}
From \eqref{eq: L1}, \eqref{eq: Erho1} and \eqref{eq: Erho2}, we get
\begin{equation}
\E\left[L_1\right]\leq \frac{C}{N_n}.
\end{equation} 
The first term $L_2$ can be estimated similarly as the terms $D_i$ and $E_i$ in the proof of Proposition~\ref{prop: 2nd}. We have
\begin{align}
\E L_2&\leq 2\int_0^T \bigg\{\frac{C M_n}{M_n+N_n}\E\big(X_s^{i}-\widehat{X_s^{i}}\big)^2+\frac{C}{M_n+N_n}\sum_{k=1}^{M_n}\Bigg(\E\Big[\big(Y_s^{k}-\widehat{Y_s^{k}}\big)^2\Big]\E\Big[\big(X_s^{i}-\widehat{X_s^{i}}\big)^2\Big]\Bigg)^{\frac{1}{2}}\notag
\\&\qquad+ C\Big|(1-a)-\frac{M_n}{M_n+N_n}\Big|\Big[\E\big(X_s^{i}-\widehat{X_s^{i}}\big)^2\Big]^{\frac{1}{2}}\bigg\}\,ds\notag
\\&\leq 2\int_0^T \bigg\{\frac{C M_n}{M_n+N_n}\sup_{s\in[0,T]}\E\big(X_s^{i}-\widehat{X_s^{i}}\big)^2+ C\Big|(1-a)-\frac{M_n}{M_n+N_n}\Big|\Big[\sup_{s\in[0,T]}\E\big(X_s^{i}-\widehat{X_s^{i}}\big)^2\Big]^{\frac{1}{2}}\notag
\\&\qquad +\frac{C M_n}{M_n+N_n}\Bigg(\sup_{s\in[0,T]}\E\Big[\big(Y_s^{i}-\widehat{Y_s^{i}}\big)^2\Big]\sup_{s\in[0,T]}\E\Big[\big(X_s^{i}-\widehat{X_s^{i}}\big)^2\Big]\Bigg)^{\frac{1}{2}}\bigg\}\,ds\notag
\\&\leq 2T\bigg\{\frac{C M_n}{M_n+N_n}\sup_{s\in[0,T]}\E\big(X_s^{i}-\widehat{X_s^{i}}\big)^2+ C\Big|(1-a)-\frac{M_n}{M_n+N_n}\Big|\Big[\sup_{s\in[0,T]}\E\big(X_s^{i}-\widehat{X_s^{i}}\big)^2\Big]^{\frac{1}{2}}\notag
\\&\qquad+\frac{C M_n}{M_n+N_n}\Bigg(\sup_{s\in[0,T]}\E\Big[\big(Y_s^{i}-\widehat{Y_s^{i}}\big)^2\Big]\sup_{s\in[0,T]}\E\Big[\big(X_s^{i}-\widehat{X_s^{i}}\big)^2\Big]\Bigg)^{\frac{1}{2}}
\bigg\},
\label{eq: estimate 4}
\end{align}
According to Proposition \ref{prop: 2nd}  the RHS of \eqref{eq: estimate 4} tends to $0$ as $n\rightarrow \infty$. If $\frac{N_n}{M_n}$ is constant then it is bounded by $\frac{C}{N_n}$.
Substituting estimates of $\E\left[L_1\right]$ and $\E\left[L_2\right]$ back into \eqref{eq: estimate 3}, we obtain
\[
\E\Big[\sup_{t\in [0,T]}\big(X_t^{i}-\widehat{X_t^{i}}\big)^2\Big]\leq \frac{C e^{2\theta_1 T}}{N_n},
\]
from which the first assertion of the theorem follows. The second assertion is obtained analogously.
\end{proof}
\section{Existence and non-uniqueness of invariant measures}
\label{sec: invariant}
In this section we prove Theorem \ref{theo: invariant} showing the existence and non-uniqueness of invariant measures of the PDE system \eqref{eq: PDEs} that is given below.
\begin{subequations}
\begin{align}
\partial_t\mu_t&=\div\Big(\big(\nabla V_1+a(\nabla F_{11}\ast\mu_t)+(1-a)(\nabla F_{12}\ast \nu_t)\big)\mu_t\Big)+\frac{\sigma^2}{2}\Delta\mu_t,\\
\partial_t\nu_t&=\div\Big(\big(\nabla V_2+a(\nabla F_{21}\ast\mu_t)+(1-a)(\nabla F_{22}\ast \nu_t)\big)\nu_t\Big)+\frac{\sigma^2}{2}\Delta\nu_t.
\end{align}
\end{subequations}
We only consider the quadratic interaction potentials
$$
F_{ij}(x)=\frac{\alpha_{ij} x^2}{2}.
$$
We expect that extensions to polynomial potentials could be possible but the analysis will be much more intricate. We leave this for future investigation.
Stationary solutions $(\mu(x)\,dx,\nu(x)\,dx)$ of the above system is determined by
\begin{subequations}
\begin{align}
\label{eq: invariant measures}
\mu(x)&=\frac{\exp\Big(-\frac{2}{\sigma^2}\big(V_1(x)+a F_{11}\ast \mu(x)+(1-a)F_{12}\ast\nu(x)\big)\Big)}{\int \exp\Big(-\frac{2}{\sigma^2}\big(V_1(x)+a F_{11}\ast \mu(x)+(1-a)F_{12}\ast\nu(x)\big)\Big)\,dx},\\
\nu(x)&=\frac{\exp\Big(-\frac{2}{\sigma^2}\big(V_2(x)+a F_{21}\ast \mu(x)+(1-a)F_{22}\ast\nu(x)\big)\Big)}{\int \exp\Big(-\frac{2}{\sigma^2}\big(V_2(x)+a F_{21}\ast \mu(x)+(1-a)F_{22}\ast\nu(x)\big)\Big)\,dx}.
\end{align}
\end{subequations}
We define 
\begin{equation*}
m_1:=\int x\mu(x)\,dx\quad\text{and}\quad m_2=\int x\nu(x)\,dx.
\end{equation*}
Using explicit formulas $F_{ij}(x)=\frac{\alpha_{ij}}{2}x^2$ for $i,j=1,2$, we obtain
\begin{subequations}
\label{eq: eqn for mu and nu}
\begin{align}
\mu(x)=\frac{\exp\Big(-\frac{2}{\sigma^2}\big(V_1(x)+a\frac{\alpha_{11}}{2}x^2-a\alpha_{11} m_1 x+(1-a)\frac{\alpha_{12}}{2}x^2-(1-a)\alpha_{12} m_2 x\big)\Big)}{\int \exp\Big(-\frac{2}{\sigma^2}\big(V_1(x)+a\frac{\alpha_{11}}{2}x^2-a\alpha_{11} m_1 x+(1-a)\frac{\alpha_{12}}{2}x^2-(1-a)\alpha_{12} m_2 x\big)\Big)\,dx},\\
\nu(x)=\frac{\exp\Big(-\frac{2}{\sigma^2}\big(V_2(x)+a\frac{\alpha_{21}}{2}x^2-a\alpha_{21} m_1 x+(1-a)\frac{\alpha_{22}}{2}x^2-(1-a)\alpha_{22} m_2 x\big)\Big)}{\int \exp\Big(-\frac{2}{\sigma^2}\big(V_2(x)+a\frac{\alpha_{21}}{2}x^2-a\alpha_{21} m_1 x+(1-a)\frac{\alpha_{22}}{2}x^2-(1-a)\alpha_{22} m_2 x\big)\Big)\,dx}.
\end{align}
 \end{subequations}
Therefore, $(m_1,m_2)$ satisfies the following system

\begin{subequations}
\label{eq: m-system}
 \begin{align}
m_1&=\frac{\int x \exp\Big(-\frac{2}{\sigma^2}\big(V_1(x)+a\frac{\alpha_{11}}{2}x^2-a\alpha_{11} m_1 x+(1-a)\frac{\alpha_{12}}{2}x^2-(1-a)\alpha_{12} m_2 x\big)\Big)\,dx}{\int \exp\Big(-\frac{2}{\sigma^2}\big(V_1(x)+a\frac{\alpha_{11}}{2}x^2-a\alpha_{11} m_1 x+(1-a)\frac{\alpha_{12}}{2}x^2-(1-a)\alpha_{12} m_2 x\big)\Big)\,dx},\label{eq: m-system1}\\
m_2&=\frac{\int x \exp\Big(-\frac{2}{\sigma^2}\big(V_2(x)+a\frac{\alpha_{21}}{2}x^2-a\alpha_{21} m_1 x+(1-a)\frac{\alpha_{22}}{2}x^2-(1-a)\alpha_{22} m_2 x\big)\Big)}{\int \exp\Big(-\frac{2}{\sigma^2}\big(V_2(x)+a\frac{\alpha_{21}}{2}x^2-a\alpha_{21} m_1 x+(1-a)\frac{\alpha_{22}}{2}x^2-(1-a)\alpha_{22} m_2 x\big)\Big)\,dx}.\label{eq: m-system2}
 \end{align}
\end{subequations}
We define $\Phi_1(m_1,m_2)$ and $\Phi_2(m_1,m_2)$
to be the right-hand sides of \eqref{eq: m-system1} and \eqref{eq: m-system2}, respectively. Setting $\Phi(m_1,m_2):=(\Phi_1,\Phi_2)(m_1,m_2)$. We rewrite \eqref{eq: m-system} as
\begin{equation}
\label{eq: m-system3}
(m_1,m_2)=\Phi(m_1,m_2),
\end{equation}
where $\Phi(m_1,m_2)$ denotes its right-hand side. 
\subsection{Symmetrical invariant measure}
We suppose that $V_1$ and $V_2$ are symmetrical. 
\begin{lemma} There exists a unique pair of symmetric invariant measures that are given by
\begin{subequations}
\begin{align}
\mu^0(x)&=\frac{\exp\Big(-\frac{2}{\sigma^2}\big(V_1(x)+a\frac{\alpha_{11}}{2}x^2+(1-a)\frac{\alpha_{12}}{2}x^2\big)\Big)}{\int \exp\Big(-\frac{2}{\sigma^2}\big(V_1(x)+a\frac{\alpha_{11}}{2}x^2+(1-a)\frac{\alpha_{12}}{2}x^2\big)\Big)\,dx},\\
\nu^0(x)&=\frac{\exp\Big(-\frac{2}{\sigma^2}\big(V_2(x)+a\frac{\alpha_{21}}{2}x^2+(1-a)\frac{\alpha_{22}}{2}x^2\big)\Big)}{\int \exp\Big(-\frac{2}{\sigma^2}\big(V_2(x)+a\frac{\alpha_{21}}{2}x^2+(1-a)\frac{\alpha_{22}}{2}x^2\big)\Big)\,dx}.
\end{align}
\end{subequations}
\end{lemma}
\begin{proof}
Since $(\mu^0,\nu^0)$ satisfy \eqref{eq: invariant measures} with their mean values $(0,0)$ that fulfill \eqref{eq: m-system2}, they are invariant measures and are symmetric because $V_1$ and $V_2$ are symmetric. Now suppose that $(\mu^0,\nu^0)$ is an arbitrary symmetric invariant measure. Then $(\widehat{\mu}^0,\widehat{\nu}^0)$  satisfies \eqref{eq: invariant measures} with $(m_1,m_2)$ replaced by their mean values $(m_{\widehat{\mu}^0}, m_{\widehat{\nu}^0})$. Since  $(\widehat{\mu}^0$ and $\widehat{\nu}^0)$ are symmetric, we have
\[
m_{\widehat{\mu}^0}=\int x \widehat{\mu}^0(x)\,dx=0=\int x \widehat{\nu}^0(x)\,dx=m_{\widehat{\nu}^0}.
\]
Substituting these values back into \eqref{eq: invariant measures} we obtain $(\widehat{\mu}^0,\widehat{\nu}^0)=(\mu^0,\nu^0)$.
\end{proof}
\subsection{Other invariant measures}
We are now interested in non-symmetrical invariant measures.
\begin{assumption}
\label{asspt: common minizer}
Suppose that $V_1$ and $V_2$ have a common unique minimizer $m^\ast$
\[
V_1'(m^\ast)=V_2'(m^\ast)=0,~~ V_1''(m^\ast)>0~~\text{and}~~V_2''(m^\ast)>0.
\]
\end{assumption}

We will make use of the following result.
\begin{lemma}\cite[Lemma A.3]{HerrmannTugaut10}
\label{lem: HT10}
Let $U$ and $G$ be two $C^\infty(\mathbb{R})$-continuous functions. Let $\lambda$ be a parameter 
that belongs to some compact interval $\mathcal{I}$ of $\R$. We define $U_\lambda=U+\lambda G$. Suppose that $U_\lambda(z)\geq z^2$ for $|z|$ larger than some value $R$ independent of $\lambda$ and that $U_\lambda$ has a unique global minimum at $z_\lambda$ with $U_\lambda''(z_\lambda)>0$. Let $f_m$ be a $C^3$-continuous function depending on some parameter $m$ that belongs to a compact set $\mathcal{M}$. Furthermore, we also assume that there exists some constant $\theta>0$ such that $|f_m(z)|\leq \exp[\theta|U_\lambda(z)|]$ for all $z\geq R, \lambda\in\mathcal{I}, m\in\mathcal{M}$ and $f_m^{(k)}$ is locally bounded uniformly with respect to the parameter $m\in\mathcal{M}$ for $0\leq k\leq 3$. Let $a,b\in \bar{\R}$ such that $a<z_\lambda<b$. Then the following asymptotic result holds as $\varepsilon$ tends to $0$: 
\begin{equation}
\int_a^b f_m(z)\exp\Big[-\frac{2U_\lambda(z)}{\varepsilon}\Big]\,dz=\sqrt{\frac{\pi \varepsilon}{\mathcal{U}_2}} \exp\Big[-\frac{2U_\lambda(z_\lambda)}{\varepsilon}\Big]\Big\{f_m(z_\lambda)+\gamma_0(\lambda)\varepsilon +o^{(1)}_{\mathcal{I}\mathcal{M}}(\varepsilon)\Big\},
\end{equation}
with
\begin{equation}
\gamma_0(\lambda)=f_m(z_\lambda)\Big(\frac{5\U_3^2}{48\U_2^3}-\frac{\U_4}{16\U_2^2}\Big)-f_m'(z_\lambda)\frac{\U_3}{4\U_2^2}+\frac{f_m''(z_\lambda)}{4\U_2}.
\end{equation}
Here $\U_k=U_\lambda^{(k)}(z_\lambda)$ and $o^{(1)}_{\mathcal{I}\mathcal{M}}(\varepsilon)/\varepsilon$ converges to $0$ as $\varepsilon$ goes to $0$ uniformly with respect to the parameters $m$ and $\lambda$. Moreover, for any $n\geq 1$, we have
\begin{equation}
\frac{\int_{\R}z^n e^{f_m(z)} e^{-\frac{2U_\lambda(z)}{\varepsilon}}\,dz}{\int_{\R} e^{f_m(z)} e^{-\frac{2U_\lambda(z)}{\varepsilon}}\,dz}-z_\lambda^n\approx -\frac{n z_\lambda^{n-2}}{4\U_2}\Big[z_\lambda\frac{\U_3}{\U_2}-n+1-2z_\lambda f_m'(z_\lambda)\Big]\varepsilon,
\end{equation}
where the estimate is uniform with respect to the parameters $m$ and $\lambda$ as $\varepsilon\rightarrow 0$.
\end{lemma}
We are now ready to prove Theorem \ref{theo: invariant}.
\begin{proof}[Proof of Theorem \ref{theo: invariant}]

We recall that $\rho>0$ is defined such that
\begin{equation}
\label{eq: rho}
\rho\geq \max\Big\{\frac{|V_1^{(3)}(m^\ast)|}{4V_1''(m^\ast)(V_1''(m^\ast)+a\alpha_{11}+(1-a)\alpha_{12})},\frac{|V_2^{(3)}(m^\ast)|}{4V_2''(m^\ast)(V_2''(m^\ast)+a\alpha_{21}+(1-a)\alpha_{22})} \Big\}.
\end{equation} 
We define
\[
D(\sigma):=[m^\ast-\rho\sigma^2,m^\ast+\rho\sigma^2]\times[m^\ast-\rho\sigma^2,m^\ast+\rho\sigma^2].
\]
Let $(m_1,m_2)\in D(\sigma)$. Then there exist $\rho_1,\rho_2$ with $0\leq |\rho_1|,\rho_2|\leq \rho$ such that
\[
m_i=m^\ast+\rho_i \sigma^2,~ i=1,2.
\]
We have
\begin{align*}
&\Phi_1(m_1,m_2)
\\&=\frac{\int x \exp\Big(-\frac{2}{\sigma^2}\big(V_1(x)+a\frac{\alpha_{11}}{2}x^2-a\alpha_{11} (m^\ast+\rho_1\sigma^2) x+(1-a)\frac{\alpha_{12}}{2}x^2-(1-a)\alpha_{12} (m^\ast+\rho_2\sigma^2) x\big)\Big)\,dx}{\int \exp\Big(-\frac{2}{\sigma^2}\big(V_1(x)+a\frac{\alpha_{11}}{2}x^2-a\alpha_{11} (m^\ast+\rho_1\sigma^2) x+(1-a)\frac{\alpha_{12}}{2}x^2-(1-a)\alpha_{12} (m^\ast+\rho_2\sigma^2) x\big)\Big)\,dx}
\\&=\frac{\int x e^{2 a\alpha_{11}\rho_1 x+2(1-a)\alpha_{12}\rho_2 x} \exp\Big(-\frac{2}{\sigma^2}\big(V_1(x)+a\frac{\alpha_{11}}{2}x^2-a\alpha_{11} m^\ast x+(1-a)\frac{\alpha_{12}}{2}x^2-(1-a)\alpha_{12} m^\ast x\big)\Big)\,dx}{\int  e^{2 a\alpha_{11}\rho_1 x+2(1-a)\alpha_{12}\rho_2 x} \exp\Big(-\frac{2}{\sigma^2}\big(V_1(x)+a\frac{\alpha_{11}}{2}x^2-a\alpha_{11} m^\ast x+(1-a)\frac{\alpha_{12}}{2}x^2-(1-a)\alpha_{12} m^\ast x\big)\Big)\,dx}.
\end{align*}
Set $U(x)=V_1(x)+a\frac{\alpha_{11}}{2}x^2-a\alpha_{11} m^\ast x+(1-a)\frac{\alpha_{12}}{2}x^2-(1-a)\alpha_{12} m^\ast x$. We have
\begin{align*}
U'(x)&=V_1'(x)+(a\alpha_{11}+(1-a)\alpha_{12})(x-m^\ast),
\\ U''(x)&=V_1''(x)+a\alpha_{11}+(1-a)\alpha_{12},
\\ U^{(3)}(x)&=V_1^{(3)}(x).
\end{align*}
Since $U'(m^\ast)=V_1'(m^\ast)=0, U''(m^\ast)=V_1''(m^\ast)>0$, $m^\ast$ is the unique minimizer of $U$. Applying Lemma \ref{lem: HT10} for $f(x)= 2 a\alpha_{11}\rho_1 x+2(1-a)\alpha_{12}\rho_2 x, n=1, U(x)$ and $\lambda=0$, we get
\begin{align*}
\Phi_1(m^\ast+\rho_1\sigma^2,m^\ast+\rho_2\sigma^2)&=m^\ast-\frac{1}{4m^\ast \U_2}\Big[m^\ast\frac{\U_3}{\U_2}-2m^\ast f'(m^\ast)\Big]\,\sigma^2+o(\sigma^2)
\\&=m^\ast-\bigg[\frac{V_1^{(3)}(m^\ast)}{4(V_1''(m^\ast)+\tau)^2}+\frac{\zeta}{V_1''(m^\ast)+\tau}\bigg]\,\sigma^2+o(\sigma^2)
\\&=: m^\ast-k_1\sigma^2+o(\sigma^2),
\end{align*}
where
\[
\tau:=a\alpha_{11}+(1-a)\alpha_{12},\quad \zeta:=a\alpha_{11}\rho_1+(1-a)\alpha_{12}\rho_2.
\]
We have 
\begin{align*}
|k_1|&\leq \frac{|V_1^{(3)}|}{4(V_1''(m^\ast)+\tau)^2}+\frac{a\alpha_{11}|\rho_1|+(1-a)\alpha_{12}|\rho_2|}{V_1''(m^\ast)+\tau}
\\&\leq \frac{|V_1^{(3)}|}{4(V_1''(m^\ast)+\tau)^2}+\frac{a\alpha_{11}\rho+(1-a)\alpha_{12}\rho}{V_1''(m^\ast)+\tau}
\\&=\frac{|V_1^{(3)}|}{4(V_1''(m^\ast)+\tau)^2}+\frac{\tau \rho}{V_1''(m^\ast)+\tau} 
\\&\overset{\eqref{eq: rho}}{\leq} \rho.
\end{align*}
Similarly we have
\[
\Phi_2(m^\ast+\rho_1\sigma^2,m^\ast+\rho_2\sigma^2)=m^\ast-k_2\sigma^2+o(\sigma^2)\quad\text{where}\quad |k_2|\leq \rho.
\]
Thus for $\sigma$ small enough, we have $\Phi(m^\ast+\rho_1\sigma^2,m^\ast+\rho_2\sigma^2)\in D(\sigma)$. By Brouwer's fixed-point theorem, there exist $(m_1,m_2)\in D(\sigma)$ that satisfy \eqref{eq: m-system}, thus the measures $\mu$ and $\nu$ defined in \eqref{eq: eqn for mu and nu} are invariant measures for the coupled MV-equations. 
\end{proof}
\begin{remark}
Assumption \ref{asspt: common minizer} has been used to obtain that the two functions
\begin{align*}
U(x)&=V_1(x)+a\frac{\alpha_{11}}{2}x^2-a\alpha_{11} m^\ast x+(1-a)\frac{\alpha_{12}}{2}x^2-(1-a)\alpha_{12} m^\ast x,
\\ \hat{U}(x)&=V_2(x)+a\frac{\alpha_{21}}{2}x^2-a\alpha_{21} m^\ast x+(1-a)\frac{\alpha_{22}}{2}x^2-(1-a)\alpha_{22} m^\ast x,
\end{align*}
have the common unique minimizer $m^\ast$ which is also the minimizer of $V_1$ and $V_2$. We expect that this assumption can be removed. To this end, one would need to find a solution $(m_1^\ast,m_2^\ast)$ to the following system
\begin{align*}
V_1'(m_1)+(1-a)\alpha_{12}(m_1-m_2)&=0,
\\V_2'(m_2)+a\alpha_{21}(m_1-m_2)&=0.
\end{align*}
Then one apply Brouwer's fixed-point theorem for $D(\sigma)=[m_1^\ast+\rho_1\sigma^2,m_2^\ast+\rho_2\sigma^2]$ where $0\leq |\rho_1|,|\rho_2|\leq \rho$ with a suitable choice of $\rho$.
\end{remark}
\bibliographystyle{alpha}
\bibliography{ref}
\end{document}